\documentclass[smallextended]{svjour3}
\usepackage{epsfig}
\usepackage{amsmath}
\usepackage{amssymb}
\usepackage[usenames]{color}
\usepackage{appendix}
\newcommand{\rd}{\mathrm{d}}
\newcommand{\R}{\mathbb{R}}

\newcommand{\Z}{\mathbb{Z}}
\newcommand{\C}{\mathbb{C}}
\newcommand{\ri}{{\mathrm{i}}}
\newcommand{\re}{{\mathrm{e}}}
\newcommand{\erfc}{\mathrm{erfc}}

\begin{document}
\title{Computing Fresnel Integrals via Modified Trapezium Rules}
\author{Mohammad Alazah \and \mbox{Simon~N.~Chandler-Wilde} \and Scott~La~Porte}
\authorrunning{M.~Alazah, S.~N.~Chandler-Wilde and S.~La Porte}
\institute{Mohammad Alazah \and Simon N.~Chandler-Wilde \at Department of Mathematics and Statistics, University of Reading, Whiteknights, PO Box 220, Reading RG6 6AX, UK\\\email{m.a.m.alazah@pgr.reading.ac.uk}\\\email{S.N.Chandler-Wilde@reading.ac.uk}
\and
Scott La Porte \at Department of Mathematical Sciences, John Crank Building,
Brunel University,
Uxbridge UB8 3PH, UK\\\email{scottis@ntlworld.com}}
\dedication{Dedicated to David Hunter on the occasion of his 80th birthday}
\date{\today}
\maketitle
\begin{abstract}
In this paper we propose methods for computing Fresnel integrals based on truncated trapezium rule approximations to integrals on the real line, these trapezium rules  modified to take into account poles of the integrand near the real axis. Our starting point is a method for computation of the error function of complex argument due to Matta and Reichel ({\em J. Math. Phys.} {\bf 34} (1956), 298--307) and Hunter and Regan ({\em Math. Comp.} {\bf 26} (1972), 539--541). We construct approximations which we prove are exponentially convergent as a function of $N$, the number of quadrature points, obtaining explicit error bounds which show that accuracies of $10^{-15}$ uniformly on the real line are achieved with $N=12$, this confirmed by computations. The approximations we obtain are attractive, additionally, in that they maintain small relative errors for small and large argument, are analytic on the real axis (echoing the analyticity of the Fresnel integrals), and are straightforward to implement.  

\end{abstract}



\section{Introduction} \label{sec intro}
Let $C(x)$, $S(x)$, and $F(x)$ be the Fresnel integrals defined by
\begin{equation} \label{CSdef}
C(x) := \int_0^x \cos\left({\textstyle\frac{1}{2}}\pi t^2\right)\, \rd t, \quad S(x) := \int_0^x \sin\left({\textstyle\frac{1}{2}}\pi t^2\right)\, \rd t,
\end{equation}
and
\begin{align} \label{Fdef}
F(x) :=\frac{\re^{-\ri\pi/4}}{\sqrt{\pi}}\int^{\infty}_{x} \re^{\ri t^2} \,\rd t.
\end{align}
Our definitions in \eqref{CSdef} are those of \cite{AS} and \cite[\S7.2(iii)]{NIST}, and $F$, $C$ and $S$ are related through
\begin{equation} \label{inter}
\sqrt{2} \,\re^{\ri \pi/4} F(x) = {\textstyle \frac{1}{2}}-C\left(\sqrt{2/\pi}\, x\right) + \ri \left({\textstyle\frac{1}{2}}-S\left(\sqrt{2/\pi}\,x\right)\right).
\end{equation}

In this paper we derive new methods for computing these Fresnel integrals $F(x)$, $C(x)$ and $S(x)$. The derivation of our approximations makes use of the relationship between the Fresnel integral and the error function, that
\begin{equation} \label{Frelw}
F(x) = {\textstyle\frac{1}{2}} \mathrm{erfc}(\re^{-\ri\pi/4}x) = {\textstyle\frac{1}{2}} \, \re^{\ri x^2} \, w\left(\re^{\ri \pi/4}x\right)
\end{equation}
where $\mathrm{erfc}$ is the complementary error function, defined by
$$
\mathrm{erfc}(z) := \frac{2}{\sqrt{\pi}}\, \int_z^\infty \re^{-t^2} \rd t,
$$
and
$$
w(z) := \re^{-z^2} \mathrm{erfc}(-\ri z).
$$
It also depends on the integral representation \cite[(7.1.4)]{AS} that
\begin{equation} \label{wref}
w(z) = \frac{\ri}{\pi} \int_{-\infty}^\infty \frac{\re^{-t^2}}{z-t} \, \rd t = \frac{\ri z}{\pi} \int_{-\infty}^\infty  \frac{\re^{-t^2}}{z^2-t^2} \, \rd t, \; \mbox{ for } \mathrm{Im}(z)>0.
\end{equation}
Combining \eqref{Frelw} and \eqref{wref} gives an integral representation for $F(x)$, that
\begin{equation} \label{Fint}
F(x) = \frac{x}{2\pi}\, \re^{\ri(x^2+\pi/4)} \int_{-\infty}^\infty \frac{\re^{-t^2}}{x^2 + \ri t^2} \, \rd t, \; \mbox{ for } x >0.
\end{equation}

Fresnel integrals arise in applications throughout science and engineering, especially in problems of wave diffraction and scattering ({\em e.g.}, \cite[\S8.2]{BSU69}, \cite{CWHL12}), so that methods for the efficient and accurate computation of these functions are of wide application. The purpose of this paper is to present new approximations for the Fresnel integrals, based on $N$-point trapezium rule approximations to the integral representation \eqref{Fint} for $F(x)$, these trapezium rules modified to take into account the poles of the integrand. These poles lie near the path of integration when $x$ is small.

The observation that the trapezium rule is exponentially convergent when applied to integrals of the form
\begin{equation} \label{pc}
\int_{-\infty}^\infty \re^{-t^2} \,f(t)\, \rd t,
\end{equation}
with $f(t)$ analytic in a strip surrounding the real axis,
dates back at least to Turing \cite{Turing} and Goodwin \cite{Goodwin}. The derivation of this result uses contour integration and Cauchy's residue theorem; see \S\ref{sec approx F} below. Applying the trapezium rule with step-length $h>0$ to \eqref{Fint} leads to the approximation
\begin{equation} \label{Ftrap}
F(x) \approx\frac{xh}{\pi}\, \re^{\ri(x^2+\pi/4)}\,  \sum_{k=1}^\infty \frac{\re^{-\tau_k^2}}{x^2 + \ri \tau_k^2}, \; \mbox{ for } x >0,
\end{equation}
where
\begin{equation} \label{tauk}
\tau_k := (k-1/2)h.
\end{equation}
When  $x>0$ is large this approximation is very accurate. Indeed, if we choose
\begin{equation} \label{hchoice}
h= \sqrt{\pi/(N+1/2)}
\end{equation}
for some large integer $N$, then this approximation is essentially identical to the approximation $F_N(x)$ for $F(x)$ that we propose in \eqref{FNdef} below.
However, the approximation \eqref{Ftrap} becomes increasingly poor as $x>0$ approaches zero.

In the context of developing methods for evaluating the complementary error function of complex argument (by \eqref{Frelw}, evaluating $F(x)$ for $x$ real is just a special case of this larger problem), Chiarella and Reichel \cite{CR}, Matta and Reichel \cite{MR}, and Hunter and Regan \cite{HR} proposed modifications of the trapezium rule that follow naturally from the contour integration argument used to prove that the trapezium rule is exponentially convergent. The most appropriate form of this modification is that in \cite{HR} where the modified trapezium rule approximation
\begin{equation} \label{Ftrapm}
F(x) \approx \frac{xh}{\pi}\, \re^{\ri(x^2+\pi/4)}\,  \sum_{k=1}^\infty \frac{\re^{-\tau_k^2}}{x^2 + \ri \tau_k^2} + R(h,x), \; \mbox{ for } x >0,
\end{equation}
is proposed. Here the correction term $R(h,x)$ is defined by
$$
R(h,x) := \left\{\begin{array}{ll}
                   1/(\exp(2\pi \re^{-\ri\pi/4}x/h)+1), & \mbox{ if } 0<x < \sqrt{2}\,\pi/h, \\
                   0.5/(\exp(2\pi \re^{-\ri\pi/4}x/h)+1), & \mbox{ if } x = \sqrt{2}\,\pi /h, \\
                   0, &  \mbox{ if } x > \sqrt{2}\,\pi/h.
                 \end{array}\right.
$$
The approximation \eqref{Ftrapm} clearly coincides with $F_N(x)$, given by \eqref{FNdef}, for $0<x < \sqrt{2}\,\pi/h$, if the range of summation in \eqref{Ftrapm} is truncated to $1,...,N$ and the choice \eqref{hchoice} for $h$ is made. Hunter and Regan prove that the magnitude of the error in \eqref{Ftrapm} is
\begin{equation} \label{HRbound}
\leq \frac{x\re^{-\pi^2/h^2}}{\sqrt{\pi} \left(1- \re^{-2\pi^2/h^2}\right) \, \left|x^2/2-\pi^2/h^2\right|},
\end{equation}
for $x>0$, provided $x\neq \sqrt{2}\pi/h$. Similar estimates, it appears arrived at independently, are derived by Mori \cite{MM}, in which paper the emphasis is on computing $\mathrm{erfc}(x)$ for real $x$.

The approximation \eqref{Ftrapm} is the starting point for the method we propose in this paper. Our main contributions (see \S\ref{sec Results} for detail) are: (i) to point out that the approximation proposed in \eqref{Ftrapm} for $0<x<\sqrt{2}\, \pi/h$ in fact provides an accurate (and real-analytic) approximation to the entire function $F$ on the whole real line; (ii) to provide an optimal formula for the choice of the step-size $h$ as a function of $N$, the number of terms retained in the sum in \eqref{Ftrapm}; (iii) to prove that, with this choice of $h$, the resulting approximations are exponentially convergent as a function of $N$, uniformly on the real line (this in contrast to \eqref{HRbound} which blows up at $x=\sqrt{2}\,\pi/h$).

\subsection{Other methods for computing Fresnel integrals} \label{sec Other}
Naturally, there exist already a number of effective schemes for computation of Fresnel integrals, and we briefly summarise now the best of these. An effective computational method for smaller values of $|x|$ is to make use of the power series for $C(x)$ and $S(x)$ (see \eqref{CSpow} below).
These converge for all $x$, and very rapidly for smaller $x$, and so are widely used for computation. For example, the algorithm in the standard reference \cite{NR} uses these power series for $|x|\leq 1.5$. For this range, after the first two terms, these series are alternating series of monotonically decreasing terms, and the error in truncation has magnitude smaller than the first neglected term. Thus, for $| x|\leq 1.5$, the errors in computing $C(x)$ and $S(x)$ by these power series truncated to $N$ terms are $\leq 2\times 10^{-16}$ and $\leq 2.3\times 10^{-17}$, respectively, for $N=14$.

For$|x|> 1.5$, \cite{NR} recommends computation using the representations in terms of $\mathrm{erfc}$ which follow from \eqref{inter} and \eqref{Frelw}, and the continued fraction representation for $\re^{z^2}\mathrm{erfc}(z)=w(\ri z)$ given as \cite[(7.9.2)]{NIST}. Methods for evaluation of $w(z)$ based on continued fractions for larger complex $z$ (which can be used to evaluate $F(x)$ and hence $C(x)$ and $S(x)$) are also discussed in Gautschi \cite{Gautschi} and are finely tuned to form TOMS ``Algorithm 680'' in Poppe and Wijers \cite{PW1,PW2}. This algorithm achieves relative errors of $10^{-14}$ over ``nearly all'' the complex plane by Taylor expansions of degree up to 20 in an ellipse around the origin, convergents of up to order 20 of continued fractions outside a larger ellipse, and a more expensive mix of Taylor expansion and continued fraction calculations in between.

Weideman \cite{JA} presents an alternative method of computation (the derivation starts from the integral representation \eqref{wref}) which approximates $w(z)$, for $\mathrm{Im}(z)>0$, by the polynomial
\begin{equation} \label{weid}
w_M(z) = \frac{2}{L^2+z^2}\sum_{n=0}^M a_n Z^n
\end{equation}
in the transformed variable $Z = (L+\ri z)/(L-\ri z)$.
Here $L=\sqrt{M/\sqrt{2}}$ and the coefficients $a_n$ can be viewed as Fourier coefficients and efficiently computed by the FFT. We will see in \S\ref{sec NR} that a polynomial degree $M=36$ in \eqref{weid} suffices to compute $F(x)=\re^{\ri x^2} w(\re^{\ri\pi/4}x)/2$ with relative error $\leq 10^{-15}$ uniformly on the positive real axis. Weideman \cite{JA} argues carefully and persuasively that, for intermediate values of $|z|$ (values in approximately the range $1.5\leq |z| \leq 5$ for the case $\arg(z)=\pi/4$ which we require), and as measured by operation counts, the work required to compute $w(z)$ to $10^{-14}$ relative accuracy is much smaller for the approximation \eqref{weid} than for Algorithm 680 \cite{PW2}.

All the approximations described above are polynomial or rational approximations (or piecewise polynomial/rational approximations, proposing different approximations on different regions). Many other authors describe approximations of these types for computing the Fresnel integrals specifically with real arguments. The best of these in terms of accuracy is Cody \cite{Cody}, where numerical coefficient values are given for  piecewise rational approximations to $C(x)$ and $S(x)$ for $0\leq x\leq 1.6$, and for piecewise rational approximations to the related functions $f(x)$ and $g(x)$ (see \eqref{Crep} and \eqref{Srep} below), for $x\geq 1.6$. These approximations, in their respective regions of validity, achieve relative errors $\leq 10^{-15.58} \approx 2.7\times 10^{-16}$, this using rational approximations which are ratios of polynomials of degree $\leq 6$; in total five different approximations are used on different subintervals of the real axis. Single rational approximations, based on a ``polar'' version of \eqref{Crep} and \eqref{Srep}, are computed in \cite{Heald}, but these are of limited accuracy (absolute errors $\leq 4\times 10^{-8}$).

\subsection{Summary of the main results} \label{sec Results}
The main result of this paper is to derive, with rigorous error bounds, a new family of approximations to $F(x)$ based on modified trapezium rules, given by
\begin{eqnarray} \label{FNdef}
F_{N}(x) & :=&	\frac{1}{2} + \frac{\ri}{2} \tan\left(A_Nx\re^{\ri\pi/4}\right) + \frac{x}{A_N}\,\re^{\ri(x^2+\pi/4)} \,\sum_{k=1}^N \frac{\re^{-t_{k}^{2}}}{ x^2+\ri t_{k}^{2}}\\ \label{FNdef2}
& = & \frac{1}{\exp\left(2A_N x\re^{-\ri\pi/4}\right)+1} + \frac{x}{A_N}\,\re^{\ri(x^2+\pi/4)} \,\sum_{k=1}^N \frac{\re^{-t_{k}^{2}}}{ x^2+\ri t_{k}^{2}},
\end{eqnarray}
where
\begin{equation} \label{tkdef}
t_{k}:=\frac{\left(k-1/2\right)\pi}{\sqrt{\left(N+1/2\right)\pi}}, \quad A_N := t_{N+1} = \sqrt{(N+1/2)\pi}.
\end{equation}
The corresponding approximations to $C(x)$ and $S(x)$ that we propose (obtained by substituting in \eqref{inter} and separating real and imaginary parts) are
\begin{eqnarray} \nonumber
C_{N}(x)&:=&	\frac{1}{2}\, \frac{\sinh{(\sqrt{\pi}\,A_N\,x)}+\sin{(\sqrt{\pi}\,A_N\,x)}}{\cos(\sqrt{\pi}\,A_N\,x)+\cosh(\sqrt{\pi}\,A_N\,x)}\\ \label{CNdef}
& & \quad + \frac{\sqrt{\pi}\, x}{A_N}\left( a_N\left(\frac{\pi}{2}x^2\right)\sin\left(\frac{\pi}{2}x^2\right)- b_N\left(\frac{\pi}{2}x^2\right)\cos\left(\frac{\pi}{2}x^2\right) \right)
\end{eqnarray}
and
\begin{eqnarray} \nonumber
S_{N}(x)&:=&	\frac{1}{2}\, \frac{\sinh{(\sqrt{\pi}\,A_N\,x)}-\sin{(\sqrt{\pi}\,A_N\,x)}}{\cos(\sqrt{\pi}\,A_N\,x)+\cosh(\sqrt{\pi}\,A_N\,x)}\\ \label{SNdef}
& & \quad - \frac{\sqrt{\pi}\, x}{A_N}\left( a_N\left(\frac{\pi}{2}x^2\right)\cos\left(\frac{\pi}{2}x^2\right)+ b_N\left(\frac{\pi}{2}x^2\right)\sin\left(\frac{\pi}{2}x^2\right) \right),
\end{eqnarray}
where
\begin{equation} \label{akdef}
	a_{N}(s) := s\sum_{k=1}^N \frac{\re^{-t_{k}^{2}}}{s^2+t_{k}^{4}}, \quad b_{N}(s) := \sum_{k=1}^N \frac{t_k^2\,\re^{-t_{k}^{2}}}{s^2+t_{k}^{4}}.
\end{equation}

These approximations, designed for computation of $F(x)$, $C(x)$ and $S(x)$ for all $x\in \R$, are attractive in several respects.
\begin{itemize}
\item[$\bullet$] The approximation $F_N$ is proven in Theorems \ref{thm:main_abs_bound} and \ref{thm:rel} to converge to $F$ approximately in proportion to $\exp(-\pi N)$, uniformly on the real line with respect to both absolute and relative error, and this predicted rate of exponential convergence is observed in numerical experiments (see \S\ref{sec NR}).
\item[$\bullet$] The approximations $F_N(z)$, $C_N(z)$ and $S_N(z)$ to the entire functions $F$, $C$, and $S$, are analytic in the strip $|\mathrm{Im}(z)|<\sqrt{(N+1/2)\pi/2}$  and the error bounds we prove extend in modified form into this strip. This implies exponentially convergent error estimates, presented in \S\ref{subsec:complex} and \S\ref{sec approx CS}, for the difference between the coefficients in the Maclaurin series of $F$, $C$, and $S$ and those in the corresponding series for $F_N$, $C_N$ and $S_N$.
In turn (see \S\ref{sec approx CS}), this implies that the approximations all retain small relative error for $|x|$ small, and the computations in \S\ref{sec NR} demonstrate this.
\item[$\bullet$] These approximations inherit symmetries of the Fresnel integrals. In particular, our normalisation of $F(x)$ is such that
\begin{equation} \label{symm}
F(-x) = 1- F(x),
\end{equation}
so that, in particular, $F(0)=1/2$. It is clear from \eqref{FNdef} that the same holds for $F_N(x)$, {\em i.e.},
\begin{equation} \label{symmN}
F_N(-x) = 1- F_N(x).
\end{equation}
Similarly, where an overline denotes a complex conjugate,
\begin{equation} \label{conj}
\overline{F(z)} = F(\ri\bar z) \mbox{ and }  \overline{F_N(z)} = F_N(\ri\bar z).
\end{equation}
Both these symmetries can be deduced from the structure of $C$ and $S$ and their approximations: by inspection of \eqref{CNdef} and \eqref{SNdef} we see that
\begin{equation} \label{CSNstr}
C_N(x) = x f_C(x^4), \quad S_N(x) = x^3 f_S(x^4),
\end{equation}
where $f_C$ and $f_S$ are analytic in a neighbourhood of the real line and are real-valued for real arguments. This is the same structure as $C$ and $S$ (see \eqref{CSpow}). In particular, \eqref{CSNstr} implies that $C_N$ and $S_N$, like $C$ and $S$, are odd functions.
\item[$\bullet$] These approximations are straightforward to code. Tables \ref{matlab_code} and \ref{matlab_codeCS} show the short Matlab codes used to evaluate $F_N$, $C_N$ and $S_N$ for all the computations in this paper. 
\end{itemize}

\begin{table}
\small
\begin{verbatim}
function f = fresnel(x,N)
% Evaluates the approximation F_N(x) to the Fresnel integral F(x).
% x is a real scalar or matrix,
% N is the positive integer controlling accuracy (suggest N=12),
% f is the corresponding scalar or matrix of values of F_N(x).
select = x>=0;
f = zeros(size(x));
if any(select), f(select) = F(x(select),N); end
if any(~select), f(~select) = 1-F(-x(~select),N); end

function f = F(x,N)
h = sqrt(pi/(N+0.5));
t = h*((N:-1:1)-0.5);  AN = pi/h;
t2 = t.*t; t4 = t2.*t2; et2 = exp(-t2);
rooti = exp(i*pi/4);
z = rooti*x; x2 = x.*x; x4 = x2.*x2; z2 = i*x2;
S = (-et2(1)./(x4+t4(1))).*(z2+t2(1));
for n = 2:N
    S = S + (-et2(n)./(x4+t4(n))).*(z2+t2(n));
end
ez = exp((2*AN*i*rooti)*x);
f =  (i/AN)*z.*exp(z2).*S + ez./(ez+1);

\end{verbatim}
\normalsize
\caption{Matlab code to evaluate $F_N(x)$ given by \eqref{FNdef2}, making use of \eqref{symmN} for $x<0$.}
\label{matlab_code}
\end{table}

We end this introduction by outlining the remainder of the paper. In \S\ref{sec approx F} we derive the approximation \eqref{FNdef} to $F(x)$ and prove rigorous bounds on $|F(x)-F_N(x)|$. In \S\ref{sec approx CS} we deduce from this the approximations \eqref{CNdef} and \eqref{SNdef} and bounds on the errors $C(x)-C_N(x)$ and $S(x)-S_N(x)$, especially bounds for $x$ small. In \S\ref{sec NR} we show numerical results, comparing our new approximations with the error bounds derived in the earlier sections and with certain rival methods for computing Fresnel integrals. The appendix proves what appears to be a new, sharp lower bound on $|\erfc(z)|$, for $\mathrm{Re}(z)\geq 0$, of some independent interest, potentially useful for deriving rigorous upper bounds on the relative error in approximate methods for computing $\erfc$ (e.g., the methods of \cite{HR,PW1,JA}). The relevance of this lower bound to the rest of the paper is that it implies, via \eqref{Frelw}, a new lower bound on $|F(x)|$ for $x>0$, of independent interest and a key component in our theoretical bounds on relative errors in \S\ref{sec approx F}.

\section{The Approximation for $F(x)$ and its Error Bounds} \label{sec approx F}
In this section we derive the approximation $F_N(x)$ to $F(x)$ and derive error bounds for this approximation demonstrating that both absolute and relative errors converge exponentially to zero as $N$ increases, uniformly on the real line, and that $N=12$ is enough to achieve errors $< 10^{-15}$.
The first part of our derivation follows in large part Matta and Reichel \cite{MR} and Hunter and Regan \cite{HR}. From \eqref{Fint} we have that, for $x>0$,
\begin{equation} \label{FinI}
I := \int_{-\infty}^\infty f(t) \, \rd t = F(x), \mbox{ where } f(t) := \re^{\ri(x^2+\pi/4)}\,\frac{x}{2\pi}\, \frac{\re^{-t^2}}{x^2 + \ri t^2},
\end{equation}
and we have suppressed in our notation the dependence of $f(t)$ on $x$.

Given $h>0$ let
$$
g(z) = \ri \tan(\pi z/h),
$$
which is an odd meromorphic function with simple poles at the points $\tau_k$, defined by \eqref{tauk}, which has the property that, for $z=X+\ri H$ with $X\in\R$, $H>0$,
\begin{equation} \label{gbound}
|1+g(z)| \leq \frac{2\re^{-2\pi H/h}}{1- \re^{-2\pi H/h}}.
\end{equation}

The approximation \eqref{Ftrapm} is obtained by considering the integral in the complex plane
\begin{align}\label{J}
J = \int_{\Gamma} f(z)(1+g(z))\,\rd z,
\end{align}
where the path of integration is from $-\infty$ to $\infty$ along the real axis, except that the path makes small semicircular deformations to pass above each of the simple poles at the points $\tau_k$, $k\in \Z$. Explicitly, the $k$th deformation is the semicircle $\gamma_k = \{\tau_k + \epsilon \re^{-\ri\theta}:\pi\leq \theta \leq 2\pi\}$, with $\epsilon$ in the range $(0,h/2)$ small enough so that the simple pole singularity in $f(z)$ at $z=z_0:=\re^{\ri\pi/4} x$ lies above $\Gamma$. Then, since $f(z)g(z)$ is an odd function, we see that
$$
J = \int_{\Gamma} f(z)\,\rd z + \int_{\Gamma} f(z)g(z)\,\rd z = I + \sum_{k\in \Z} \int_{\gamma_k} f(z)g(z)\,\rd z.
$$
In the limit $\epsilon\to 0$, $\int_{\gamma_k} f(z)g(z)\, \rd z \to -\pi \ri\, \mathrm{Res}(fg,\tau_k) = -h f(\tau_k)$, where $\mathrm{Res}(fg,\tau_k)$ denotes the residue of $fg$ at $\tau_k$. Thus $J = I-I_h$, where
\begin{equation} \label{Ih}
I_h = h\sum_{k\in\Z} f(\tau_k) = 2h\sum_{k=1}^\infty f((k-1/2)h)
\end{equation}
is a trapezium/midpoint rule approximation to $I$.

For $H>0$ let
$$
J_H = \int_{\Gamma_H} f(z)(1+g(z))\,\rd z,
$$
where the path of integration $\Gamma_H$ is the line $\mathrm{Im}(z)=H$, traversed in the direction of increasing $\mathrm{Re}(z)$. It follows from Cauchy's residue theorem that
\begin{equation}\label{prop1}
J-J_H = \mathbf{H}\left(\sqrt{2}\,H-x\right)\, PC_h,
\end{equation}
where $\mathbf{H}$ is the Heaviside step function (defined by $\mathbf{H}(t) = 1$, for $t>0$, $\mathbf{H}(0)=1/2$, and $\mathbf{H}(t)=0$, for $t<0$), and
$$
PC_h = 2\pi \ri \,\mathrm{Res}(f(1+g), z_0) = \frac{1}{2}\,  \left(1+ g(z_0)\right) = \frac{1}{2}\,  \left(1+ \ri \tan\left(\re^{\ri\pi/4}x\pi/h\right)\right).
$$
Thus
\begin{equation} \label{main}
I= I_h + \mathbf{H}\left(\sqrt{2}\,H-x\right)\, PC_h + J_H.
\end{equation}
The point of this formula is that $I_h + \mathbf{H}\left(\sqrt{2}\,H-x\right)\, PC_h$ is a computable approximation to $I$ and the integral $J_H$ is small, as quantified in the following proposition.
\begin{proposition}
Let $e_h$ denote the value of the integral $J_H$ when we choose $H=\pi/h$. Then, for $x>0$,
\begin{equation} \label{main2}
|e_h| \leq \delta_1(x) := \frac{x\,\,\re^{-\pi^2/h^2}}{\sqrt{\pi}\,|\pi^2/h^2-x^2/2|\, \left(1- \re^{-2\pi^2 /h^2}\right)}.
\end{equation}
\end{proposition}
\begin{proof}
For $z=X+\ri H$,
$$
|x^2+\ri z^2|=|z_0-z|\, |z_0+z| \geq |x/\sqrt{2}-H|\, |x/\sqrt{2}+H| = |x^2/2-H^2|
$$
so, using \eqref{gbound} and recalling that $\int_{-\infty}^\infty \re^{-t^2}\, dt = \sqrt{\pi}$ , we see that
\begin{equation} \nonumber
\left| J_H\right| \leq \frac{x\,\re^{H^2-2\pi H/h}}{\sqrt{\pi}\, |H^2-x^2/2|\, \left(1- \re^{-2\pi H/h}\right)}.
\end{equation}
Choosing $H=\pi/h$, to minimise the exponent $H^2-2\pi H/h$, the result \eqref{main2} follows. \qed
\end{proof}

Note that $I_h + \mathbf{H}\left(\sqrt{2}\,\pi/h-x\right)\, PC_h = I_h + R(h,x)$ is precisely the approximation \eqref{Ftrapm}, and that the above bound on $e_h$ is precisely the bound \eqref{HRbound} from \cite{HR}.

\begin{theorem} \label{thm:Delta_h}
Let $I_h^* := I_h + PC_h$ and $e_h^* := I-I_h^*$. Then, for $x>0$,
\begin{equation}
|e_h^*| \leq \Delta_h(x),
\end{equation}
where
\begin{equation} \label{Deldef}
\Delta_h(x) := \left\{\begin{array}{cc}
                      \delta_1(x), & 0\leq \frac{x}{\sqrt{2}} \leq \frac{3}{4}\frac{\pi}{h}, \\
                       \delta_2(x), & \frac{3}{4}\frac{\pi}{h}< \frac{x}{\sqrt{2}} < \frac{5}{4}\frac{\pi}{h}, \\
                        \delta_3(x), & \frac{x}{\sqrt{2}} \geq \frac{5}{4}\frac{\pi}{h}.
                    \end{array}\right.
\end{equation}
Here $\delta_1$ is defined by \eqref{main2},
\begin{equation} \label{b4}
\delta_2(x) := \frac{4hx\,\,\re^{-\pi^2/h^2}}{\sqrt{\pi}\,\pi(\pi/h+x/\sqrt{2})\, \left(1- \re^{-2\pi^2 /h^2}\right)}\left(1 + 2\sqrt{\pi}\,\re^{-\beta\pi^2/h^2}\right),
\end{equation}
with $\beta = 1- \sqrt{2}/2-(2\sqrt{2}+1)/16 \approx 0.0536$, and
\begin{equation} \label{b2}
\delta_3(x) := \delta_1(x) + \frac{\re^{-\sqrt{2}\,\pi x/h}}{1- \re^{-\sqrt{2}\,\pi x/h}}.
\end{equation}
\end{theorem}
\begin{proof}
The bound \eqref{main2} implies that $
|e_h^*| \leq \delta_1(x)$, for $0<x<\sqrt{2}\,\pi/h$.
Since, applying \eqref{gbound},
$$
|PC_h| \leq \frac{\re^{-\sqrt{2}\,\pi x/h}}{1- \re^{-\sqrt{2}\,\pi x/h}},
$$
the bound \eqref{main2} also implies that $|e_h^*| \leq \delta_3(x)$, for $x>\sqrt{2}\,\pi/h$.

Setting $H=\pi/h$, select $\epsilon$ in the range $(0,H)$ and consider the case that
$\left|x/\sqrt{2}-H\right| < \epsilon$.
In this case we observe that the derivation of \eqref{main} can be modified to show that
\begin{equation} \label{ehs}
e_h^* = \int_{\Gamma^*_H} f(z)(1+g(z))\,\rd z,
\end{equation}
where the contour $\Gamma^*_H$ passes above the pole in $f$ at $z_0$; precisely, $\Gamma_H^*$ is the union of $\Gamma^\prime$ and $\gamma$, where $\Gamma^\prime = \{z\in \Gamma_H: |z-z_0|>\epsilon\}$ and $\gamma$ is the circular arc $\gamma = \{z_0+\epsilon \re^{\ri\theta}:\theta_0\leq \theta \leq \pi-\theta_0\}$, where $\theta_0 = \sin^{-1}((H-x/\sqrt{2})/\epsilon)\in (-\pi/2,\pi/2)$. For $z\in \Gamma^\prime$ it holds that
\begin{equation} \label{b3}
|x^2+\ri z^2|=|z_0-z|\, |z_0+z| \geq \epsilon\, |x/\sqrt{2}+H|.
\end{equation}
Thus, and applying \eqref{gbound}, similarly to \eqref{main2} we deduce that
\begin{equation} \label{bound1}
\left|\int_{\Gamma^\prime} f(z)(1+g(z))\,\rd z\right| \leq \frac{x\,\,\re^{-\pi^2/h^2}}{\sqrt{\pi}\,\epsilon|\pi/h+x/\sqrt{2}|\, \left(1- \re^{-2\pi^2 /h^2}\right)}.
\end{equation}
To bound the integral over $\gamma$ we note that, for $z=X+\ri Y = z_0+\epsilon \re^{\ri\theta}\in \gamma$, \eqref{b3} is true and $Y\geq H$. Further,
$| \re^{-z^2}| = \re^P$, where
$$
P = Y^2-X^2 = 2x\epsilon\sin(\theta-\pi/4) - \epsilon^2\cos(2\theta)< 2x\epsilon+ \epsilon^2 \leq 2\sqrt{2}H\epsilon + (2\sqrt{2}+1)\epsilon^2,
$$
since $\left|x/\sqrt{2}-H\right| < \epsilon$. From these bounds and \eqref{gbound}, defining $\alpha =\epsilon/H\in (0,1)$, we deduce that
\begin{equation} \label{bound2a}
\left|\int_{\gamma} f(z)(1+g(z))\,\rd z\right| \leq \frac{2x\,\exp((2\sqrt{2}\alpha+(2\sqrt{2}+1)\alpha^2-2)\pi^2/h^2)}{\epsilon|\pi/h+x/\sqrt{2}|\, \left(1- \re^{-2\pi^2 /h^2}\right)}.
\end{equation}
For $x$ in the range $\left|x/\sqrt{2}-H\right| < \epsilon$ we can bound $e_h^*$ using \eqref{ehs}, \eqref{bound1}, \eqref{bound2a}, and the triangle inequality, to get that
\begin{equation} \label{b4a}
|e_h^*|\leq \frac{hx\,\,\re^{-\pi^2/h^2}}{\alpha \sqrt{\pi}\,\pi|\pi/h+x/\sqrt{2}|\, \left(1- \re^{-2\pi^2 /h^2}\right)}\left(1 + 2\sqrt{\pi}\,\re^{-\beta\pi^2/h^2}\right),
\end{equation}
where $\beta = 1-2\sqrt{2}\alpha-(2\sqrt{2}+1)\alpha^2$. Noting that $\beta >0$ if and only if $0<\alpha<\alpha_0$, where $\alpha_0 = (1+2\sqrt{2})^{-1}\approx 0.2612$, we choose $\alpha<\alpha_0$ to be $\alpha=1/4$. With this choice it follows from \eqref{b4a} that $|e_h^*|\leq \delta_2(x)$ for $\frac{3}{4}\frac{\pi}{h}< \frac{x}{\sqrt{2}} < \frac{5}{4}\frac{\pi}{h}$, and the proof is complete. \qed
\end{proof}

The approximation $F_N(x)$, given by \eqref{FNdef}, that we propose for $I=F(x)$ is just $I_h^*=I_h+PC_h$ with a particular choice of $h$ and with the range of summation in \eqref{Ih} reduced to the finite range $1,...,N$. This induces an additional error,
\begin{equation} \label{TNdef}
T_N := 2h\sum_{m=N+1}^\infty f(\tau_m),
\end{equation}
that we bound in the next proposition.
\begin{proposition} \label{prop:TN}
For $x>0$,
\begin{eqnarray*}
|T_N|\leq\frac{(2h\tau_{N+1}+ 1)x}{2\pi \tau_{N+1}\sqrt{x^4+\tau_{N+1}^4}} \, \re^{-\tau_{N+1}^2}.
\end{eqnarray*}
\end{proposition}
\begin{proof}
\begin{eqnarray*}
|T_N| & \leq &\frac{hx}{\pi} \sum_{m=N+1}^\infty \frac{\re^{-\tau_m^2}}{\sqrt{x^4+\tau_m^4}}\\
& \leq & \frac{x}{2\pi\sqrt{x^4+\tau_{N+1}^4}} \left(2h\re^{-\tau_{N+1}^2} + 2h\sum_{m={N+2}}^\infty \re^{-\tau_m^2}\right)\\
& \leq & \frac{x}{2\pi\sqrt{x^4+\tau_{N+1}^4}} \left(2h\re^{-\tau_{N+1}^2} + 2\int_{\tau_{N+1}}^\infty \re^{-t^2}\rd t\right)\\
& \leq & \frac{x}{2\pi\sqrt{x^4+\tau_{N+1}^4}} \left(2h\re^{-\tau_{N+1}^2} + \frac{\re^{-\tau_{N+1}^2}}{\tau_{N+1}}\right)= \frac{(2h\tau_{N+1}+ 1)x}{2\pi \tau_{N+1}\sqrt{x^4+\tau_{N+1}^4}} \, \re^{-\tau_{N+1}^2}.
\end{eqnarray*}
To arrive at the last line we have used that, for $x>0$,
\begin{equation} \label{erfcb}
2\int_x^\infty \re^{-t^2}\rd t = \frac{\re^{-x^2}}{x}-\int_x^\infty \frac{\re^{-t^2}}{t^2}\rd t< \frac{\re^{-x^2}}{x}.
\end{equation}\qed
\end{proof}

 At this point we make a choice of $h$ to approximately equalise $\Delta_h(x)$ in Theorem \ref{thm:Delta_h} (which is approximately proportional to $\exp(-\pi^2/h^2)$)  and the bound on $T_N$ in Proposition \ref{prop:TN}, choosing $h$ so that $\pi/h= \tau_{N+1}=(N+1/2)h$. In other words, we make the choice $h=\sqrt{\pi/(N+1/2)}$ given by \eqref{hchoice}, in which case $\tau_{N+1}=A_N=\sqrt{(N+1/2)\pi}$, and $\tau_k=t_k$, where $t_k$ is defined by \eqref{tkdef}.
Making this choice of $h$ we see that
\begin{equation} \label{ENdef}
E_N(x) := F(x)-F_N(x) = e_h^* + T_N
\end{equation}
and that
\begin{equation} \label{TNbound}
|T_N| \leq    \frac{(2\pi+1)x}{2\pi A_{N}\sqrt{x^4+A_{N}^4}}\, \re^{-A_{N}^2}.
\end{equation}
Combining \eqref{ENdef} and \eqref{TNbound} with Theorem \ref{thm:Delta_h}, we arrive at the following theorem which is our main pointwise error bound. Theorem \ref{thm:Delta_h}, \eqref{ENdef}, and \eqref{TNbound} prove this theorem only for $x>0$, but the symmetries \eqref{symm} and \eqref{symmN} imply that
$E_N(-x)=-E_N(x)$, so that \eqref{errb} holds also for $x<0$, and, by continuity, also for $x=0$ (and in fact $E_N(0)=\eta_N(0)=0$).
\begin{theorem} \label{thm:main_pw_error}
For $x\in\R$,
\begin{equation} \label{errb}
|E_N(x)| \leq \eta_N(x):= \Delta_h(|x|) + \frac{(2\pi+1)|x|}{2\pi A_{N}\sqrt{x^4+A_{N}^4}}\, \re^{-A_{N}^2},
\end{equation}
where
\begin{equation} \label{Deldef2}
\Delta_h(x) = \left\{\begin{array}{ll}
\dfrac{x\,\re^{-A_N^2}}{\sqrt{\pi}\,(A_N^2-x^2/2)\, \left(1- \re^{-2A_N^2}\right)}, & 0\leq \dfrac{x}{\sqrt{2}} \leq \frac{3}{4}A_N, \\
\dfrac{4x\,\re^{-A_N^2}\left(1 + 2\sqrt{\pi}\,\re^{-\beta A_N^2}\right)}{\sqrt{\pi}\,A_N(A_N+x/\sqrt{2})\, \left(1- \re^{-2A_N^2}\right)}, & \frac{3}{4}A_N< \dfrac{x}{\sqrt{2}} < \frac{5}{4}A_N, \\
\dfrac{x\,\,\re^{-A_N^2}}{\sqrt{\pi}\,(x^2/2-A_N^2)\, \left(1- \re^{-2A_N^2}\right)} + \dfrac{\re^{-\sqrt{2}\,A_N x}}{1- \re^{-\sqrt{2}A_N x}} , & \dfrac{x}{\sqrt{2}} \geq \frac{5}{4}A_N.
                    \end{array}\right.
\end{equation}
\end{theorem}
We will compare $|E_N(x)|$ to the upper bound $\eta_N(x)$ for $N=9$ in Figure \ref{fig2} below. The following theorem estimates the maximum value of $\eta_N(x)$ on the real line.

\begin{theorem} \label{thm:main_abs_bound}
For $x\in\R$,
\begin{equation} \label{main3}
|F(x)-F_N(x)| \leq \eta_N(x) \leq c_N \frac{\re^{-\pi N}}{\sqrt{N+1/2}},
\end{equation}
where
\begin{align}\label{cNdef}
c_N 
&=\frac{20\sqrt{2}\re^{-\pi/2}}{9\pi\left(1- \re^{-2A_N^2}\right)} \left(1 + 2\sqrt{\pi}\,\re^{-\beta A_N^2}\right) + \frac{(2\pi+1)\re^{-\pi/2}}{2\sqrt{2}\,\pi^{3/2} A_N},
\end{align}
which decreases as $N$ increases, with
\begin{equation} \label{clim}
c_1 \approx 0.825 \;\mbox{ and }\; \lim_{N\rightarrow\infty}c_{N} = \frac{20\sqrt{2}\re^{-\pi/2}}{9\pi}\approx0.208.
\end{equation}
\end{theorem}
\begin{proof}
 It is easy to see that $\Delta_h(x)$ is increasing on $[0,\frac{5}{4}\sqrt{2}\,A_N)$ and decreasing on $[\frac{5}{4}\sqrt{2}\,A_N,\infty)$. Further,  where $\Delta_h(\frac{5}{4}\sqrt{2}\,A_N^-)$ denotes the limiting value of $\Delta_h(x)$ as $x\to \frac{5}{4}\sqrt{2}\,A_N$ from below, since $2A_N^{-1}>\re^{-A_N^2}$,
\begin{eqnarray*}
\Delta_h\left(\textstyle{\frac{5}{4}}\sqrt{2}A_N^-\right) &= &\frac{20\sqrt{2}\,\re^{-A_N^2}}{9\sqrt{\pi}\,A_N\, \left(1- \re^{-2A_N^2}\right)}\left(1 + 2\sqrt{\pi}\,\re^{-\beta A_N^2}\right)\\
&>& \frac{20\sqrt{2}\,\,\re^{-A_N^2}}{9\sqrt{\pi}\,A_N\, \left(1- \re^{-2A_N^2}\right)} + \frac{\re^{-5A^2_N/2}}{1- \re^{-5A^2_N/2}} =\Delta_h\left(\textstyle{\frac{5}{4}}\sqrt{2}\,A_N\right).
\end{eqnarray*}
Similarly, $x\Delta_h(x)$ is increasing on $[0,\frac{5}{4}\sqrt{2}\,A_N)$ and decreasing on $[\frac{5}{4}\sqrt{2}\,A_N,\infty)$. Thus, for $x\geq 0$,
\begin{equation} \label{Delbounds}
\Delta_h(x) \leq \Delta_h\left(\textstyle{\frac{5}{4}}\sqrt{2}\,A_N^-\right) \quad \mbox{ and }\quad x\Delta_h(x) \leq \textstyle{\frac{5}{4}}\sqrt{2}\,A_N\Delta_h\left(\textstyle{\frac{5}{4}}\sqrt{2}\,A_N^-\right).
\end{equation}
Moreover,
\begin{equation} \label{bbb}
\frac{|x|}{\sqrt{x^4+A_{N}^4}}\leq \frac{1}{\sqrt{2}\, A_N} \; \mbox{ and }\; \frac{x^2}{\sqrt{x^4+A_{N}^4}}< 1,  \quad \mbox{for } x\in\R.
\end{equation}
Combining \eqref{errb}, \eqref{Delbounds} and \eqref{bbb} we reach the result. \qed
\end{proof}

We can also bound the relative error in our approximation $F_{N}(x)$. The proof of Theorem \ref{lemma} is postponed to the appendix.
\begin{theorem}\label{lemma}
\begin{equation}\label{Flb1}
|F(x)|\geq\frac{1}{2+2\sqrt{\pi}\,x}, \quad \mbox{for } x\geq 0,
\end{equation}
and
\begin{equation}\label{Flb2}
|F(x)|\geq\frac{1}{2}, \quad \mbox{for } x\leq 0.
\end{equation}
\end{theorem}

\begin{theorem} \label{thm:rel}
\begin{equation} \label{main4}
\frac{|F(x)-F_N(x)|}{|F(x)|} \leq \frac{\eta_N(x)}{|F(x)|}\leq \left\{\begin{array}{ll}
                                                                     c_N^* \re^{-\pi N}, & \;\mbox{for } x\geq 0, \\
                                                                     2c_N \dfrac{\re^{-\pi N}}{\sqrt{N+1/2}}, & \;\mbox{for } x\leq 0,
                                                                   \end{array}\right.
\end{equation}
where
\begin{eqnarray} \nonumber
c_N^* 
&=& \frac{10\sqrt{2}\left(4+5\sqrt{2\pi}A_N\right)\left(1+2\sqrt{\pi}\re^{-\beta A_N^2}\right)}{9\sqrt{\pi}\, \re^{\pi/2}\, A_N\left(1-\re^{-2A_N^2}\right)} + \frac{(2\pi+1)}{\pi\re^{\pi/2} A_N }\, \left(\frac{1}{\sqrt{2}\, A_N} + \sqrt{\pi}\right),
\end{eqnarray}
which decreases as $N$ increases, with $c_1^*\approx 10.4$ and $\lim_{N\to\infty} c_N^*=100\re^{-\pi/2}/9$ $\approx 2.3$.
\end{theorem}
\begin{proof}
Combining \eqref{Flb1}, \eqref{errb}, \eqref{Delbounds}, and \eqref{bbb}, we see that, for $x\geq 0$,
$$
\frac{\eta_N(x)}{|F(x)|}\leq \left(2+\textstyle{\frac{5}{2}}\sqrt{2\pi}A_N\right)\Delta_h\left(\textstyle{\frac{5}{4}}\sqrt{2}\,A_N^-\right) + \frac{(2\pi+1)}{\pi }\, \frac{\re^{-A_{N}^2}}{A_{N}}\left(\frac{1}{\sqrt{2}\, A_N} + \sqrt{\pi}\right).
$$
This implies the bound \eqref{main4} for $x\geq 0$. The bound \eqref{main4} for $x\leq 0$ follows immediately from \eqref{Flb2} and \eqref{main3}. \qed
 \end{proof}

 In the above theorems we use \eqref{errb} and \eqref{Deldef2} to bound the maximum absolute and relative errors in the approximation $F_N(x)$. These inequalities, additionally, imply that $F_N(x)$ is particularly accurate for $|x|$ small. For $|x| \leq A_N/\sqrt{2}=\sqrt{(N+1/2)\pi/2}$, it follows from \eqref{errb} and \eqref{Deldef2} that
\begin{equation} \label{MAIN3}
|F(x)-F_N(x)| \leq \eta(x) \leq \tilde c_N |x|\frac{\re^{-\pi N}}{2N+1}
\end{equation}
where
\begin{equation} \label{tildec}
\tilde c_N = \frac{8}{3\pi^{3/2}\re^{\pi/2} \left(1- \re^{-2A_N^2}\right)}+ \frac{(2\pi+1)}{\pi^2 \re^{\pi/2} A_{N}},
\end{equation}
which decreases as $N$ increases, with $\tilde c_1 \approx 0.17 $ and $\lim_{N\to \infty} \tilde c_N = 8/(3\pi^{3/2}\re^{\pi/2})\approx 0.10$.

\subsection{Extensions of the error bounds into the complex plane} \label{subsec:complex}
In \S\ref{sec intro} we have made claims regarding the analyticity of the approximation $F_N(x)$, considered as a function of $x$ in the complex plane. We justify these claims now. One attractive feature of the modified trapezium rule approximation $I_h^*$ is that, in contrast to $I_h$, it is entire as a function of $x$. This is not immediately obvious: $I_h^*=I_h+PC_h$, and $PC_h$ has simple pole singularities at $x=\re^{-\ri\pi/4}\tau_k$, $k\in\Z$. But $I_h$ also has simple poles at the same points and it is an easy calculation to see that the residues add to zero, so that the singularities cancel out. Since $F_N(x) = I_h^* - T_N$, with $h$ given by \eqref{hchoice}, it follows that the singularities of $F_N(x)$ are those of $T_N$, {\em i.e.}, simple poles at $\pm  \re^{-\ri\pi/4}t_k$, for $k=N+1,N+2, ...$. Thus $F_N(x)$ is a meromorphic function and, in particular, is analytic in the strip $|\mathrm{Im}(x)| < A_N/\sqrt{2}$ and in the first and third quadrants of the complex plane.

We will note two consequences of this analyticity and the bounds that we have already proved. In these arguments we will use an extension of the maximum principle for analytic functions to unbounded domains, that if $w(z)$ is analytic in an open quadrant in the complex plane, let us say $Q = \{z\in \C:0< |\arg(z)| < \pi/2\}$, and is continuous and bounded in its closure, then
\begin{equation} \label{pl}
\sup_{z \in Q} |w(z)| \leq \sup_{z\in \partial Q} |w(z)|,
\end{equation}
 where $\partial Q$ denotes the boundary of the quadrant. (This sort of extension of the maximum principle to unbounded domains is due to Phragmen and Lindel\"of; see, {\em e.g.}, \cite{Rudin}.)

 The first consequence is that, from \eqref{ENdef}, \eqref{main3}, and \eqref{conj}, it follows that the bound \eqref{main3} holds on both the real and imaginary axes. Further, from \eqref{Frelw} and the asymptotics of $\erfc(z)$ in the complex plane \cite[(7.1.23)]{AS}, it follows that $F(z)\to 0$, uniformly in $\arg(z)$, for $0\leq \arg(z)\leq \pi/2$; moreover, it is clear from \eqref{FNdef2} that the same holds for $F_N(z)$ and hence for $E_N(z)$. Thus \eqref{pl} implies that \eqref{main3} holds for $0\leq \arg(z)\leq \pi/2$, and \eqref{symm} and \eqref{symmN} then imply that \eqref{main3} holds also for $\pi\leq \arg(z)\leq 3\pi/4$.

It is clear from the derivations above that, if $h$ is given by \eqref{hchoice}, then $I_h^*$ also satisfies the bound \eqref{main3}, {\em i.e.},
\begin{equation} \label{bb}
|F(z)-I_h^*| \leq c_N \frac{\re^{-\pi N}}{\sqrt{N+1/2}}\, ,
\end{equation}
this holding in the first instance for real $z$, then for imaginary $z$, and finally for all $z$ in the first and third quadrants.
The bound \eqref{main3} cannot hold in the second or fourth quadrant because $E_N(z)=F(z)-F_N(z)$ has poles there. This issue does not hold for $F(z)-I_h^*$, which is an entire function, but \eqref{bb} cannot hold in the whole complex plane because this, by Liouville's theorem \cite{Rudin}, would imply that $F(z)-I_h^*$ is a constant.
What does hold is that $\re^{-\ri z^2}(F(z)-I_h^*)$ is bounded in the second and fourth quadrants, this a consequence of the definition of $I_h^*$ and the asymptotics of $\re^{z^2}\erfc(z)$ at infinity. Thus it follows from \eqref{pl}, and since $|\re^{-\ri z^2}|=1$ if $z$ is real or pure imaginary, that
\begin{equation} \label{bb2}
|F(z)-I_h^*| \leq c_N \re^{-xy}\, \frac{\re^{-\pi N}}{\sqrt{N+1/2}}\, ,
\end{equation}
for $z=x+\ri y$ in the second and fourth quadrants.

We can use the bound \eqref{bb2} to obtain a bound on $E_N(x)$ in the second and fourth quadrants. Clearly, where $T_N$ is defined by \eqref{TNdef}, with $h$ given by \eqref{hchoice}, for $z=x+\ri y$ in the second and fourth quadrants,
\begin{equation} \nonumber
|F(z)-F_N(z)| \leq c_N \re^{-xy}\, \frac{\re^{-\pi N}}{\sqrt{N+1/2}} + |T_N|.
\end{equation}
Further, arguing as below \eqref{TNdef}, if $|y|\leq A_N/(2\sqrt{2})$ so that
$$
|z^2+\ri t_k^2| \geq \left(\frac{A_N}{\sqrt{2}}-|y|\right)\left(\left(\frac{A_N}{\sqrt{2}}-|y|\right)^2 + \left(\frac{A_N}{\sqrt{2}}+|x|\right)^2\right)\geq \frac{A_N}{2\sqrt{2}}\left(A_N^2/8 + |x|^2\right),
$$
which implies that $|z^2+\ri t_k^2| \geq |z|A_N/(2\sqrt{2})$, then
$$
|T_N| \leq \re^{-xy} \frac{(2\pi+1)\sqrt{2}}{\pi A_N^2} \, \re^{-A_N^2} = \re^{-xy} \frac{\sqrt{2}(2\pi+1)}{\pi^{3/2}\exp(\pi/2) (N+1/2)} \, \re^{-\pi N}.
$$
Thus, for $z=x+\ri y$ in the second and fourth quadrants with $|y| \leq A_N/(2\sqrt{2})$,
\begin{equation} \label{bb3}
|F(z)-F_N(z)| \leq  \hat c_N \re^{-xy}\, \frac{\re^{-\pi N}}{\sqrt{N+1/2}}
\end{equation}
where
\begin{equation} \label{hatcN}
\hat c_N := c_N + \frac{\sqrt{2}(2\pi+1)}{\pi^{3/2}\exp(\pi/2) \sqrt{N+1/2}},
\end{equation}
which is decreasing with $\hat c_1 \approx 1.14$ and $\lim_{N\to\infty} \hat c_N=  \lim_{N\to\infty} c_N\approx 0.208$.

We observe above that the bound  \eqref{main3} on $E_N(z)=F(z)-F_N(z)$ holds for all complex $z$ in the first and third quadrants of the complex plane, and on the boundaries of those quadrants, the real and imaginary axes, while the bound \eqref{bb3} holds in the second and fourth quadrants for $|\mathrm{Im}(z)|\leq A_N/(2\sqrt{2})$. These bounds imply that the coefficients in the Maclaurin series of $F_N(z)$ are close to those of $F(z)$. Precisely, at least for $|z| < A_N/\sqrt{2}$,
$$
F(z) = \sum_{n=0}^\infty a_n z^n \;\;\; \mbox{ and } \;\;\;F_N(z) = \sum_{n=0}^\infty b_n z^n,
$$
with $a_n = F^{(n)}(0)/n!$, $b_n = F_N^{(n)}(0)/n!$. Thus, where $M_N=\sup_{|z|< \sqrt{\pi/2}}|E_N(z)|$, it follows from Cauchy's estimate \cite[Theorem 10.26]{Rudin} and the bounds \eqref{main3} and \eqref{bb3} that, for $N\geq 4$ so that $A_N/(2\sqrt{2})\geq \sqrt{\pi/2}$,
\begin{equation} \label{cpeb}
|a_n - b_n| = \frac{|E_N^{(n)}(0)|}{n!} \leq M_N \left(\frac{2}{\pi}\right)^{n/2}\leq  \hat c_N \left(\frac{2}{\pi}\right)^{n/2}\,\frac{\re^{-\pi (N-1/4)}}{\sqrt{N+1/2}}.
\end{equation}

\section{Approximating $C(x)$ and $S(x)$} \label{sec approx CS}
From \eqref{inter} we see that, for $x$ real,
\begin{equation} \label{CSinF}
C(x) = \mathrm{Re}\left(\sqrt{2}\,\re^{\ri\pi/4}(\textstyle{\frac{1}{2}}-F(\sqrt{\pi/2}\,x))\right), \;\,S(x) = \mathrm{Im}\left(\sqrt{2}\,\re^{\ri\pi/4}(\textstyle{\frac{1}{2}}-F(\sqrt{\pi/2}\,x))\right).
\end{equation}
Clearly, given the approximation $F_N(x)$ to $F(x)$, these relationships can be used to generate approximations for the Fresnels integrals $C(x)$ and $S(x)$. These approximations are defined, for $x\in\R$, by
\begin{equation} \label{CSNinFN}
 \begin{array}{l}
         C_N(x) = \mathrm{Re}\left(\sqrt{2}\,\re^{\ri\pi/4}(\textstyle{\frac{1}{2}}-F_N(\sqrt{\pi/2}\,x))\right),\\
         S_N(x) = \mathrm{Im}\left(\sqrt{2}\,\re^{\ri\pi/4}(\textstyle{\frac{1}{2}}-F_N(\sqrt{\pi/2}\,x))\right),
       \end{array}
\end{equation}
and are given explicitly in \eqref{CNdef} and \eqref{SNdef}. We note the similarity between \eqref{CNdef} and \eqref{SNdef} and the formulae \cite[(7.5.3)-(7.5.4)]{NIST}
\begin{eqnarray} \label{Crep}
C(x) & = & \textstyle{\frac{1}{2}} + f(x) \sin\left( \textstyle{\frac{1}{2}}\pi x^2\right) - g(x) \cos\left( \textstyle{\frac{1}{2}}\pi x^2\right),\\ \label{Srep}
S(x) & = & \textstyle{\frac{1}{2}} - f(x) \cos\left( \textstyle{\frac{1}{2}}\pi x^2\right) - g(x) \sin\left( \textstyle{\frac{1}{2}}\pi x^2\right),
\end{eqnarray}
 which express $C(x)$ and $S(x)$ in terms of the auxiliary functions, $f(x)$ and $g(x)$, for the Fresnel integrals \cite[\S7.2(iv)]{NIST}. Indeed, it follows from \cite[(7.7.10)-(7.7.11)]{NIST} that, for $x>0$, $f(x)$ and $g(x)$ have the integral representations
 $$
 f(x) = \frac{\sqrt{\pi}\, x^3}{2} \int_0^\infty \frac{\re^{-t^2}}{\left(\frac{\pi}{2}x^2\right)^2 + t^4}\, \rd t \;\mbox{ and }\; g(x) = \frac{x}{\sqrt{\pi}} \int_0^\infty \frac{t^2\re^{-t^2}}{\left(\frac{\pi}{2}x^2\right)^2 + t^4}\, \rd t,
 $$
 and, recalling that $A_N$ is linked to the quadrature step-size through \eqref{hchoice}, it is clear that,  for $x>0$, $\sqrt{\pi}\, x a_N\left(\frac{\pi}{2}x^2\right)/A_N$ and $\sqrt{\pi}\, x b_N\left(\frac{\pi}{2}x^2\right)/A_N$ can be  viewed  as quadrature approximations to these integrals.

The approximations \eqref{CNdef} and \eqref{SNdef} inherit the accuracy of $F_N(x)$ on the real line: from \eqref{CSinF} and \eqref{CSNinFN} we see, for $x\in\R$, that
\begin{equation} \label{errorCN}
|C(x)-C_N(x)| \leq \sqrt{2} \, |E_N(\sqrt{\pi/2}\,x)| \; \mbox{ and } |S(x)-S_N(x)| \leq \sqrt{2} \, |E_N(\sqrt{\pi/2}\,x)|.
\end{equation}
where $E_N(x)=F(x)-F_N(x)$. Thus the error bounds of the previous section can be applied. In particular, from \eqref{main3} and \eqref{MAIN3} it follows that both $|C(x)-C_N(x)|$ and $|S(x)-S_N(x)|$ are
\begin{equation} \label{CSNbounds}
\leq 2c_N \frac{\re^{-\pi N}}{\sqrt{2N+1}}, \quad \mbox{for } x\in \R,
\end{equation}
and
\begin{equation} \label{CSNbounds2}
\leq \sqrt{\pi}\,\tilde c_N |x|\,\frac{\re^{-\pi N}}{2N+1}, \quad \mbox{for } |x|\leq \sqrt{N+1/2}\,.
\end{equation}
Here $c_N<0.83$ and $\tilde c_N < 0.18$ are the decreasing sequences of positive numbers defined by \eqref{cNdef} and \eqref{tildec}, respectively.

These bounds show that $C_N(x)$ and $S_N(x)$ are exponentially convergent as $N\to\infty$, uniformly on the real line, so that very accurate approximations can be obtained with very small values of $N$ (\eqref{CSNbounds} shows that both $|C_N(x)-C(x)|$ and $|S_N(x)-S(x)|$ are $\leq 1.4\times 10^{-16}$ on the real line for $N\geq 11$). In \S\ref{sec NR} we will confirm the effectiveness of these approximations by numerical experiments, checking the accuracy of \eqref{CNdef} and \eqref{SNdef} by comparison with the power series \cite[\S7.6(i)]{NIST}
\begin{eqnarray} \label{CSpow}
 C(x)=\sum_{n=0}^{\infty} \frac{(-1)^n \left(\frac{1}{2}\pi\right)^{2n}x^{4n+1}}{(2n)!(4n+1)}, \quad
S(x)=\sum_{n=0}^{\infty} \frac{(-1)^n \left(\frac{1}{2}\pi\right)^{2n+1}x^{4n+3}}{(2n+1)!(4n+3)}.
\end{eqnarray}

It follows from the analyticity of $F_N(x)$ in the complex plane, discussed in \S\ref{subsec:complex}, that $F_N(x)$ has a power series convergent in $|x|< A_N/\sqrt{2}$, and from \eqref{CSNinFN} that $C_N(x)$ and $S_N(x)$ have convergent power series representations in $|x| < A_N/\sqrt{\pi}$. From the observations below \eqref{CSNstr} it is clear that, echoing \eqref{CSpow}, these take the form
\begin{eqnarray} \label{CSNpow}
C_N(x)=\sum_{n=0}^{\infty} \mathfrak{c}_n x^{4n+1}, \quad
S_N(x)=\sum_{n=0}^{\infty} \mathfrak{s}_n x^{4n+3}.
\end{eqnarray}
Further, it follows from \eqref{CSNinFN} and \eqref{cpeb} that the coefficients $\mathfrak{c}_n$ and $\mathfrak{s}_n$ are close to the corresponding coefficients of $C(x)$ and $S(x)$, with the difference having absolute value
\begin{equation} \label{coeer}
\leq \sqrt{2}\, \hat c_N \,\frac{\re^{-\pi (N-1/4)}}{\sqrt{N+1/2}},
\end{equation}
for $N\geq 4$, where $\hat c_N\leq \hat c_4 < 0.77$ is the decreasing sequence of positive numbers given by \eqref{hatcN}.
This implies that, near zero, where $C(x)$ has a simple zero and $S(x)$ a zero of order three, the approximations $C_N(x)$ and $S_N(x)$ retain small relative error.  For $C_N(x)$ this follows already from \eqref{CSNbounds2} but to see this for $S_N(x)$ we need the stronger bound implied by \eqref{coeer} that, for $|x|<1$,
\begin{equation} \label{Ssmall}
|S(x)-S_N(x)| \leq \sqrt{2}\,  \hat c_N \,\frac{\re^{-\pi (N-1/4)}}{\sqrt{N+1/2}}\sum_{n=0}^{\infty} |x|^{4n+3} =  \frac{|x|^3}{1-|x|^4}\,\frac{\sqrt{2}\, \hat c_N\,\re^{-\pi (N-1/4)}}{\sqrt{N+1/2}}.
\end{equation}

\begin{table}
\small
\begin{verbatim}
function [C,S] = fresnelCS(x,N)
% Evaluates approximations to the Fresnel integrals C(x) and S(x).
% x is a real scalar or matrix,
% N is a positive integer controlling accuracy (suggest N=12),
% C and S are the scalars/matrices of the same size as x approximating C(x) and S(x).
h = sqrt(pi/(N+0.5));
t = h*((N:-1:1)-0.5);  AN = pi/h; rootpi = sqrt(pi);
t2 = t.*t; t4 = t2.*t2; et2 = exp(-t2);
x2pi2 = (pi/2)*x.*x; x4 = x2pi2.*x2pi2;
a = et2(1)./(x4+t4(1)); b = t2(1)*a;
for n = 2:N
    term = et2(n)./(x4+t4(n));
    a = a + term;  b = b + t2(n)*term;
end
a = a.*x2pi2;
mx = (rootpi*AN)*x; Mx = (rootpi/AN)*x;
Chalf = 0.5*sign(mx); Shalf = Chalf;
select = abs(mx)<39;
if any(select)
    mxs = mx(select); shx = sinh(mxs); sx = sin(mxs);
    den = 0.5./(cos(mxs)+cosh(mxs));
    Chalf(select) = (shx+sx).*den;
    ssdiff = shx-sx;
    select2 = abs(mxs)<1;
    if any(select2)
        mxs = mxs(select2); mxs3 = mxs.*mxs.*mxs; mxs4 = mxs3.*mxs;
        ssdiff(select2) = mxs3.*(1/3 + mxs4.*(1/2520 ...
            + mxs4.*((1/19958400)+(0.001/653837184)*mxs4)));
    end
    Shalf(select) = ssdiff.*den;
end
cx2 = cos(x2pi2); sx2 = sin(x2pi2);
C = Chalf + Mx.*(a.*sx2-b.*cx2); S = Shalf - Mx.*(a.*cx2+b.*sx2);

\end{verbatim}
\normalsize
\caption{Matlab to evaluate $C_N(x)$ and $S_N(x)$ given by \eqref{CNdef} and \eqref{SNdef}. See \S\ref{sec approx CS} for details.}
\label{matlab_codeCS}
\end{table}

Table \ref{matlab_codeCS} shows the Matlab implementing \eqref{CNdef} and \eqref{SNdef} that we use in the next section. To evaluate
$(\sinh t\pm \sin t)/(\cosh t+\cos t)$,
with $t=\sqrt{\pi}\, A_N x$, in \eqref{CNdef} and \eqref{SNdef}, we note that,
for $|t|\geq 39$, $\cosh(t)+\cos(t)$ and $\exp(t)/2$ have the same value in double precision arithmetic, as do $\sinh t\pm \sin t$ and $\mathrm{sign}(t) \exp(t)/2$. Thus this expression evaluates as $\mathrm{sign}(t)$ in double precision arithmetic for $39\leq |t|\lessapprox 710$. To avoid underflow and reduce computation time, we evaluate it as $\mathrm{sign}(t)$ for $|t|\geq 39$.
For small $t$ there is an additional issue of loss of precision in evaluating $\sinh t - \sin t$ for $|t|$ small. This is avoided in Table \ref{matlab_codeCS} by using $\sinh t - \sin t = 2t^3/3! + 2t^7/7! + \dots$ for $|t|<1$, truncating after four terms as the 5th term is negligible in double precision.

\section{Numerical Results and Comparison of Methods} \label{sec NR}

In this section we show numerical computations that confirm and illustrate the theoretical error bounds in \S\ref{sec approx F} and \S\ref{sec approx CS}, and that explore the accuracy and efficiency of our new methods, through qualitative and quantitative comparisons with certain of the other computational methods described in \S\ref{sec Other}.

\begin{figure}[h]
\begin{center}
\includegraphics[width=5.85cm]{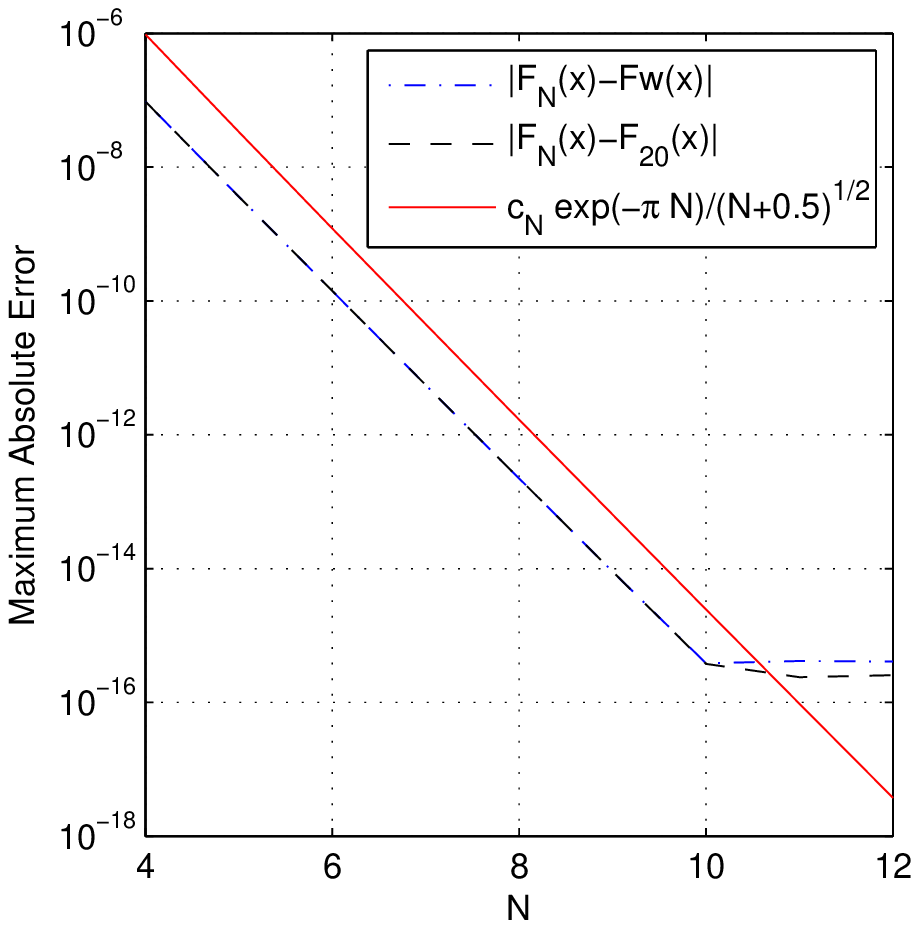}
\includegraphics[width=5.85cm]{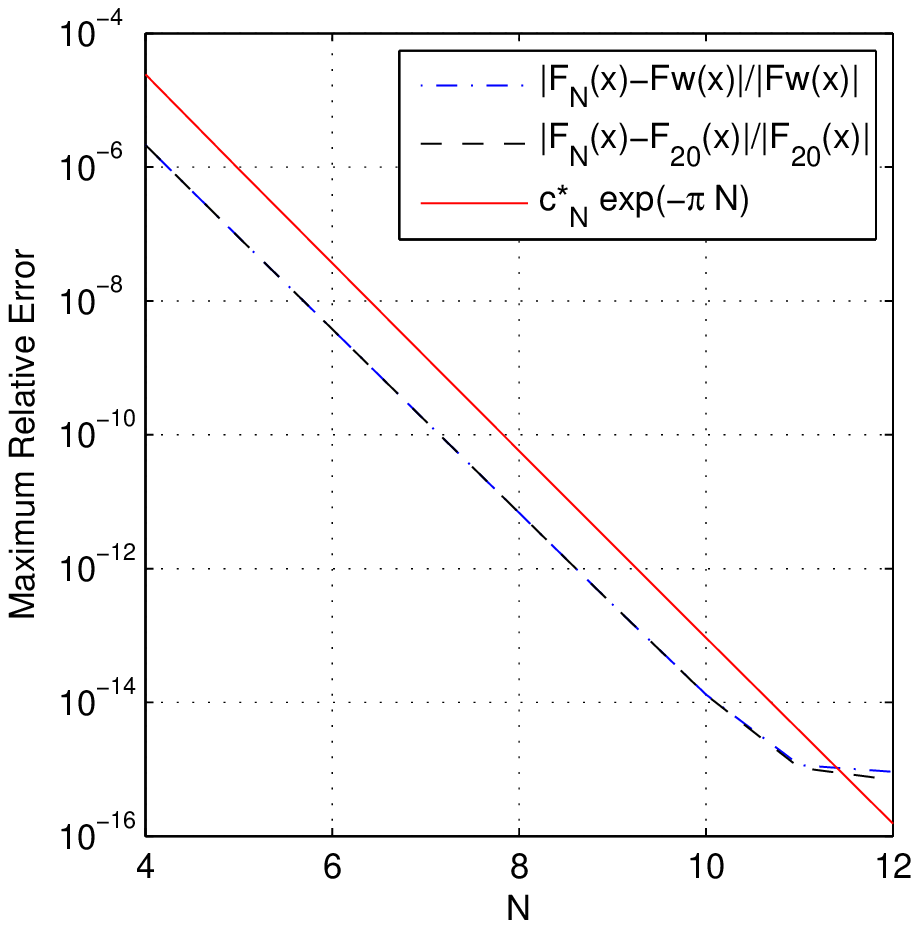}
\caption{Left hand side: maximum error, $\max_{x \geq 0}|F(x)-F_N(x)|$,  and its upper bound \eqref{main3} ({\color{red}{$-$}}), plotted against $N$, in one case where $F(x)$ is approximated by $Fw(x):=\re^{\ri x^2}w_{36}(\re^{\ri\pi/4}x)/2$ ({\color{blue}{$- \cdot - \cdot$}}) with $w_{36}(z)$ defined by \eqref{weid} and computed by the function in Table 1 of \cite{JA}, and in the other case where $F(x)$ is approximated by $F_{20}(x)$ ({\color{black}{$- - $}}). Right hand side: maximum relative error, $\max_{x \geq 0}|(F(x)-F_N(x))/F(x)|$,  and its upper bound \eqref{main4} ({\color{red}{$-$}}), plotted against $N$, where $F(x)$ is approximated in the two curves as on the left hand side. (All maximums are taken over 40,000 equally spaced points between 0 and 1,000, and all values of $F_N(x)$ are computed using the code in Table \ref{matlab_code}.)}
\label{fig1}
\end{center}
\end{figure}

In Figure \ref{fig1} it can be seen that the exponential convergence predicted by the bounds \eqref{main3} and \eqref{main4} is achieved, indeed these bounds overestimate their respective maximum errors by at most a factor of 10. Further, with $N$ as small as 12 it appears that we achieve maximum absolute and relative errors in $F_N(x)$ which are $<2.9\times 10^{-16}$ and $<9.3\times 10^{-16}$, respectively; these values are upper bounds whichever of the two methods for approximating $F(x)$ accurately is used. (We should add a note of caution here: the different approximations agree to high accuracy, but the accuracy of each approximation is limited, for large $x$, by the accuracy with which $\re^{\ri x^2}$ is computed.) These plots also verify the high accuracy of the approximation \eqref{weid} for $w(z)$ from \cite{JA}, at least for $\arg(z)=\pi/4$ and if $M$ is large enough in \eqref{weid}. Figure \ref{fig3} explores this in more detail: in each plot the trend is one of exponential convergence, but the convergence is not monotonic and is slower than that in Figure \ref{fig1}.

\begin{figure}[h]
\begin{center}
\includegraphics[width=5.85cm]{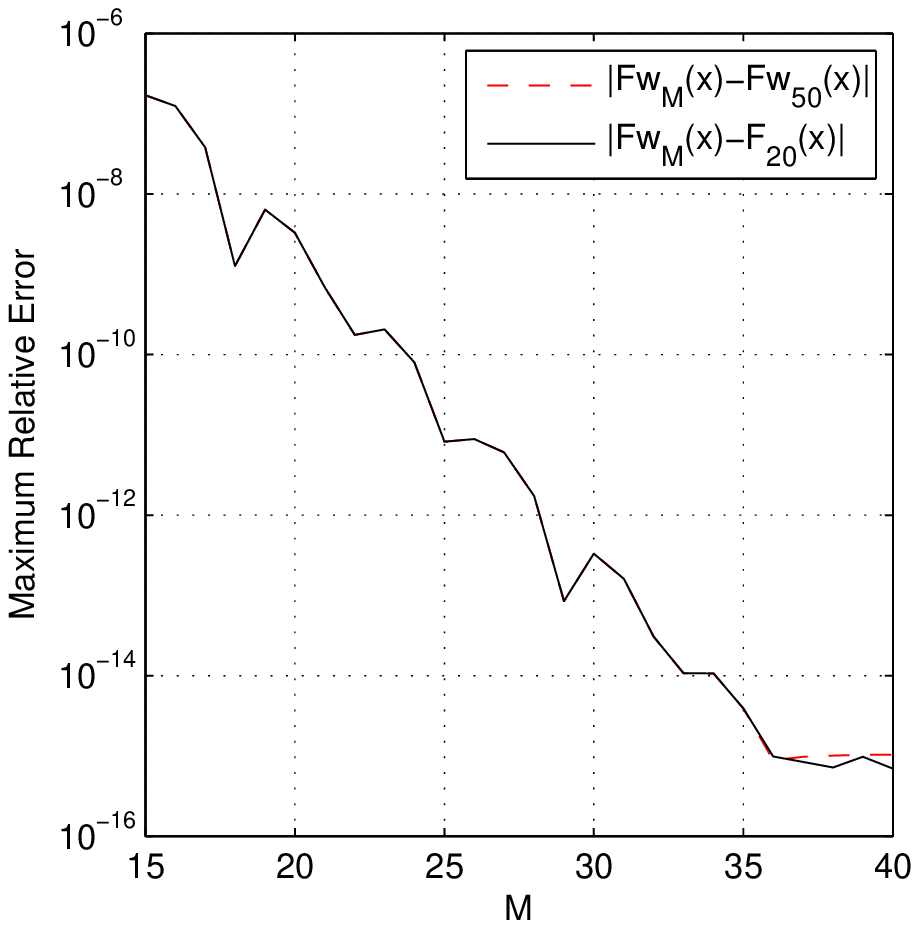}
\includegraphics[width=5.85cm]{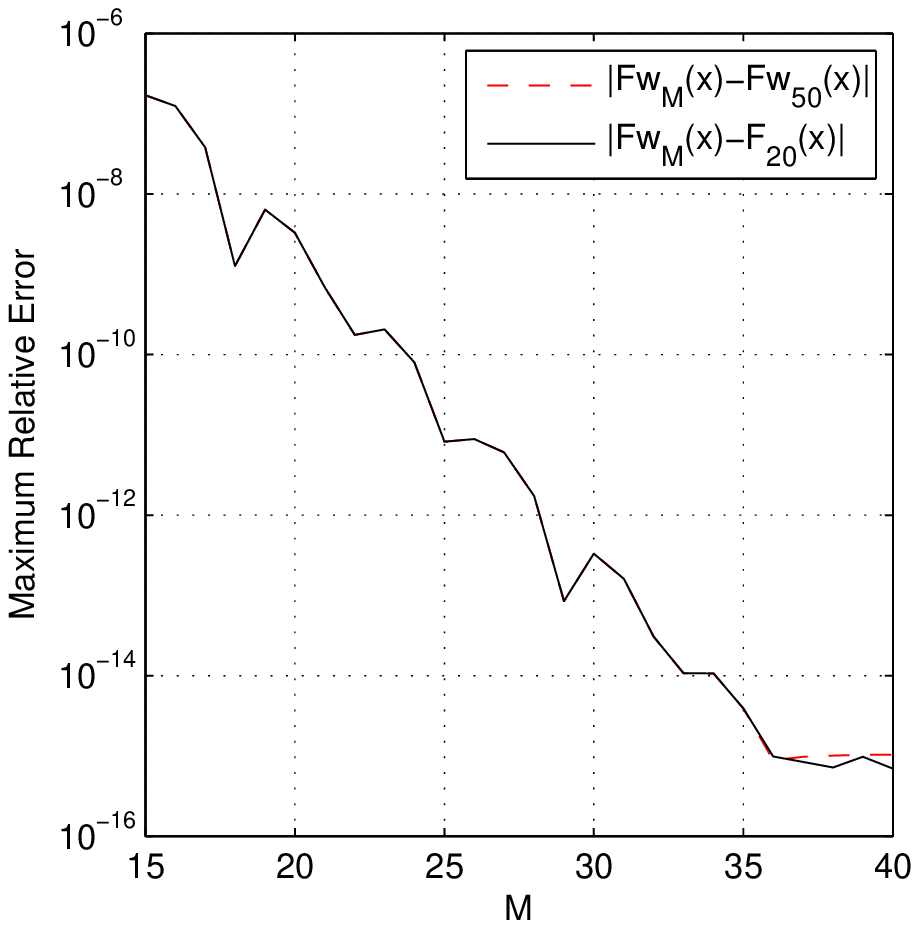}
\caption{Left hand side: maximum error, $\max_{x \geq 0}|F(x)-Fw(x)|$, where $Fw(x):=\re^{\ri x^2}w_M(\re^{\ri\pi/4}x)/2$ with $w_M(z)$ defined in \eqref{weid}. Right hand side: same, but maximum relative error, $\max_{x \geq 0}|(F(x)-Fw(x))/F(x)|$, is plotted against $M$. In each plot the two curves correspond to different methods for approximating the exact value of $F(x)$, either $F(x)\approx F_{20}(x)$ given by \eqref{FNdef} ({\color{black}{$-$}}), or $F(x)\approx Fw(x)$ with $M=50$ ({\color{red}{$- - $}}). (The maximums, as in Figure \ref{fig1}, are taken over 40,000 equally spaced points between 0 and 1,000.)}
\label{fig3}
\end{center}
\end{figure}

\begin{figure}[h]
\begin{center}
\includegraphics[width=5.85cm]{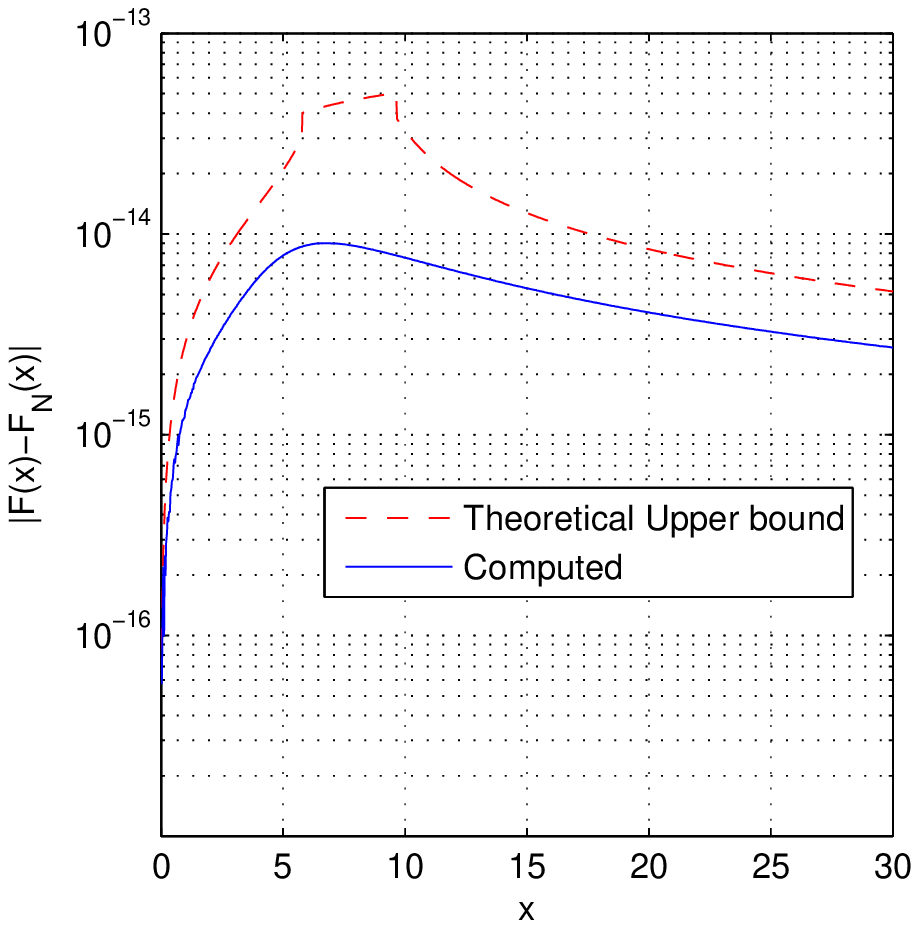}
\includegraphics[width=5.85cm]{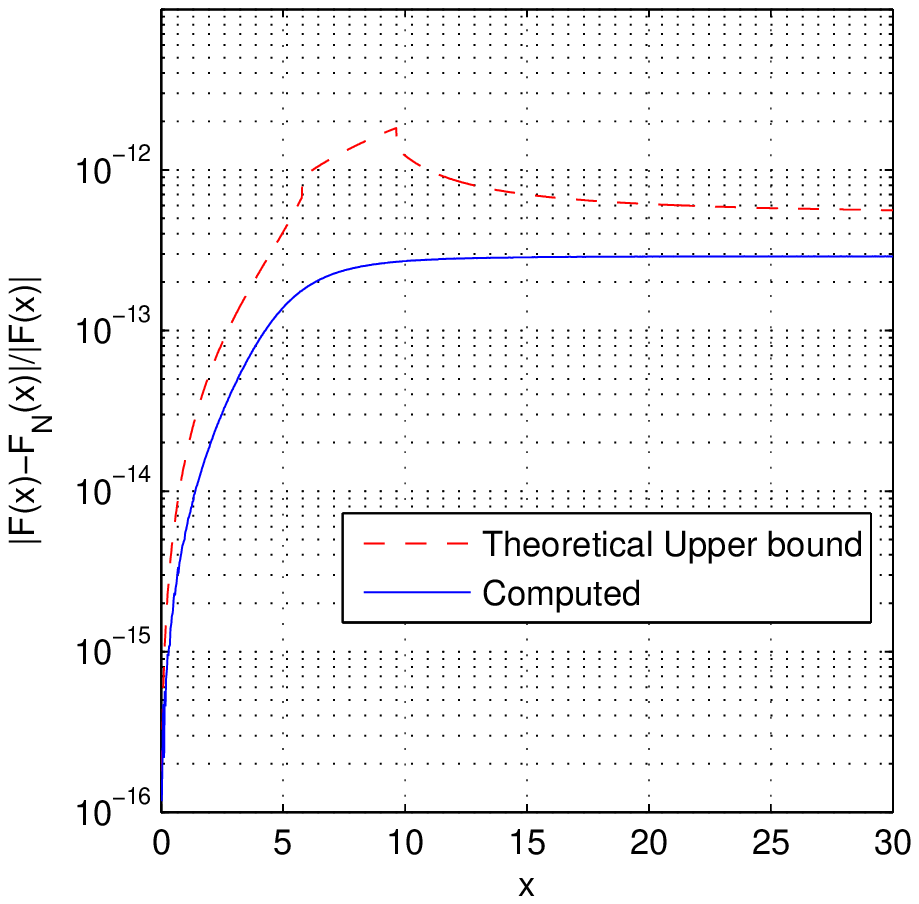}
\caption{Left hand side: absolute error, $|F(x)-F_N(x)|$ ({\color{blue}{$-$}}),  and its upper bound $\eta_N(x)$ given by \eqref{errb} ({\color{red}{$- -$}}), plotted against $x$. Right hand side: relative error, $|F(x)-F_N(x)|/|F(x)|$ ({\color{blue}{$-$}}),  and its upper bound $2(1+\sqrt{\pi}\, x)\eta_N(x)$ ({\color{red}{$- -$}}), plotted against $x$. In both plots $N=9$ and $F(x)$ is approximated by  $F_{20}(x)$.}
\label{fig2}
\end{center}
\end{figure}

In Figure \ref{fig2} we see that our pointwise theoretical error bounds are upper bounds as claimed, and that these bounds appear to capture the $x$-dependence of the errors fairly well, for example that $E_N(x)=O(x)$ as $x\to 0$, $=O(x^{-1})$ as $x\to\infty$, and that $E_N(x)$ reaches a maximum at about $x = \sqrt{2}\,A_N = \sqrt{\pi(2N+1)}$ ($\approx 7.7$ when $N=9$).

The above figures explore the accuracy of the approximation $F_N(x)$. Let us comment on efficiency. Most straightforward is a comparison of the Matlab function \verb+F(x,N)+ in Table \ref{matlab_code} with computation of $F(x)$ via the Matlab code \verb+Fw(x,M)=exp(i*x.^2).*cef(exp(i*pi/4)*x,M)/2+ that uses \verb+cef.m+ from \cite{JA} implementing \eqref{weid}. Both \verb+F(x,N)+ and \verb+cef(x,M)+ are optimised for efficiency when \verb+x+ is a large vector. The main cost in computation of $F(x)$ via \verb+cef+ when $x$ is a large vector is a complex vector exponential (for $\re^{\ri x^2}$), and the $M$ complex vector multiplications and $M$ additions required to evaluate the polynomial \eqref{weid} using Horner's algorithm.  In comparison, evaluation of $F(x)$ using \verb+F(x,N)+ in Table \ref{matlab_code} requires 2 complex vector exponentials, and slightly more than $N$ real vector multiplications/divisions, real vector additions, complex vector multiplications, and complex vector additions. From Figures \ref{fig1} and \ref{fig3} we read off that to achieve absolute and relative errors below $10^{-8}$ requires $N=6$ and $M=18$; to achieve errors below $10^{-15}$ requires $N=12$ and $M=36$. Thus computing $F(x)$ via \verb+F(x,N)+ requires a substantially lower operation count than computing via \verb+cef+. (We note, moreover, as discussed in \S\ref{sec Other} and in \S7 of \cite{JA}, that, at least for intermediate values of $x$ ($1.5\leq x \leq 5$), the operation counts for \verb+cef+ are lower than those of the method for $w(z)$ of \cite{PW1,PW2}.)

To test whether \verb+F(x,N)+ is faster we have compared computation times in Matlab (version 7.8.0.347 (R2009a) on a laptop with dual 2.4GHz P8600 Intel processors) between \verb+Fw(x,36)+ and \verb+F(x,12)+ when  \verb+x+ is a length $10^7$ vector of equally spaced numbers between 0 and 1,000. The average elapsed times were 11.1 and 15.6 seconds, respectively, so that \verb+F(x,12)+ is almost 50\% faster.

Turning to $C(x)$ and $S(x)$, in Figure \ref{fig4} we have plotted the maximum values of the absolute and relative errors in $S_N(x)$ and $C_N(x)$, computed using \verb+fresnelCS+ in Table \ref{matlab_codeCS}. As accurate values for $C(x)$ and $S(x)$ we use $C_{20}(x)$ and $S_{20}(x)$ for $x>1.5$ while, for $0<x<1.5$ (following \cite{NR}) we approximate by the series \eqref{CSpow} truncated after 15 terms, evaluated by the Horner algorithm. Exponential convergence is seen in Figure \ref{fig4}: the absolute errors are $\leq 4.5\times 10^{-16}$ for $N\geq 11$, the maximum relative error in $C_N(x)$ is $\approx 3.6\times 10^{-15}$ for $N=11$ but that in $S_N(x)$ as large as $2.7\times 10^{-13}$. These errors may be entirely acceptable, but the truncated power series \eqref{CSpow} must achieve smaller errors for small $x$ and is cheaper to evaluate. (Evaluating at $10^7$ equally spaced points between 0 and $1.5$ takes 2.9 times longer in Matlab with \verb+fresnelCS+ than evaluating 15 terms of both the series \eqref{CSpow} via Horner's algorithm.)

\begin{figure}[h]
\begin{center}
\includegraphics[width=5.85cm]{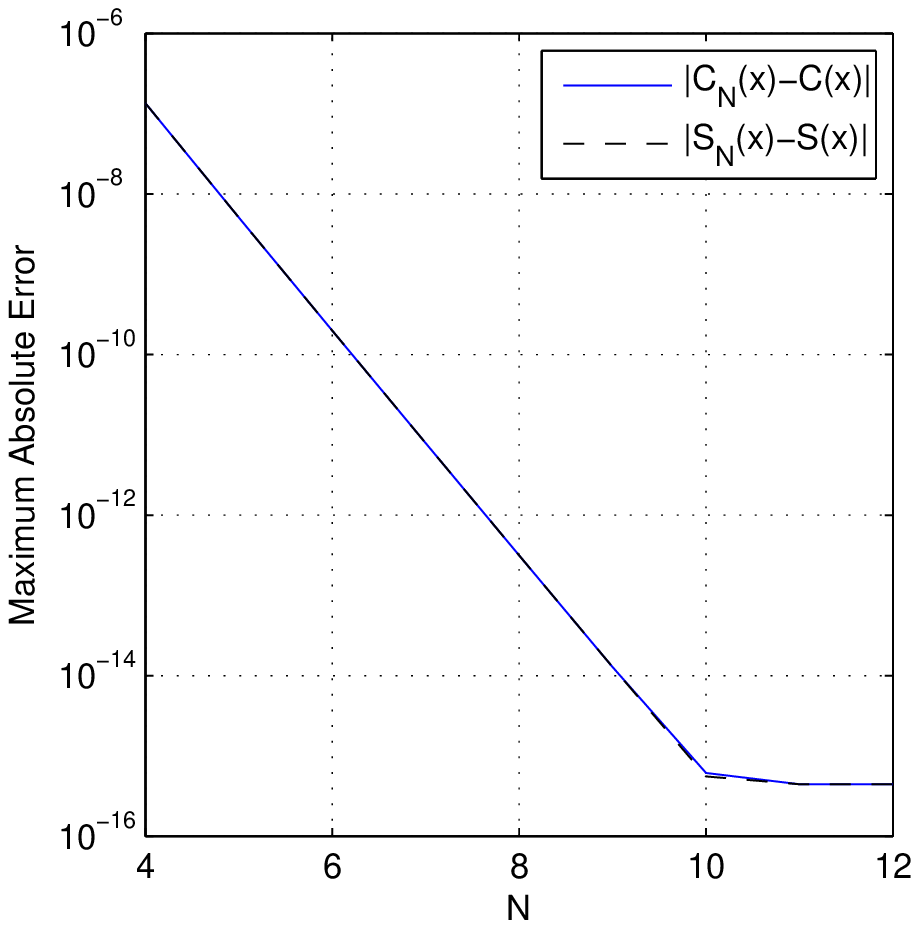}
\includegraphics[width=5.85cm]{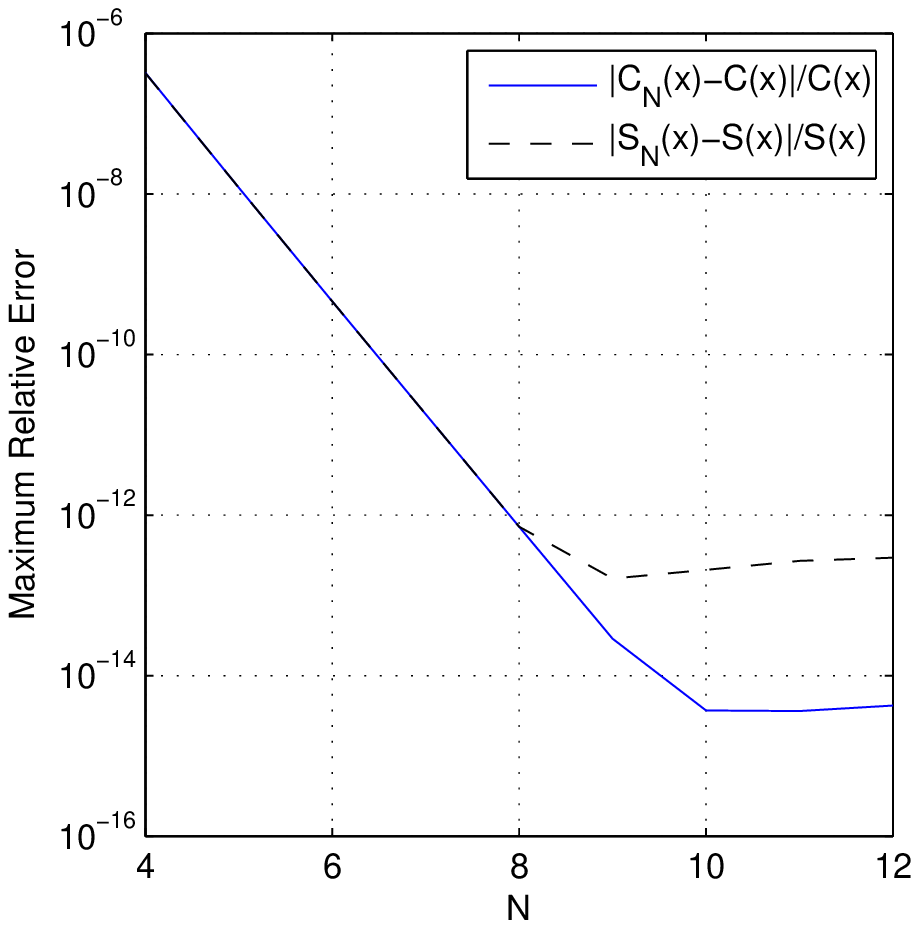}
\caption{Left hand side: maximum values of $|C_N(x)-C(x)|$ and $|S_N(x)-S(x)|$ on $0\leq x\leq 20$.  Right hand side: maximum values of $|C_N(x)-C(x)|/C(x)$ and $|S_N(x)-S(x)|/S(x)$ on $0\leq x\leq 20$.}
\label{fig4}
\end{center}
\end{figure}

\section{Concluding Remarks} \label{sec conclude}

To conclude, we have presented in this paper new approximations for the Fresnel integrals, derived from and inspired by modified trapezium rule approximations previously suggested for the complementary error function of complex argument in \cite{MR,HR}. These approximations are simple to implement (Matlab codes are included in Tables \ref{matlab_code} and \ref{matlab_codeCS}): the computation of $F_N(x)$ requires a couple of complex exponentiations and a short summation to compute a quadrature sum, and that of $C_N(x)$ and $S_N(x)$ evaluation of trigonometric and hyperbolic functions and a similar short summation. 

Operation counts and timings suggest that $F_N(x)$ with $N=12$ may be faster than previous methods, at least for intermediate values of $|x|$. In particular, the Matlab function in Table \ref{matlab_code} outperforms that in  Table 1 of \cite{JA}  for this application. The code for $S_N(x)$ and $C_N(x)$ is faster still, but the power series \eqref{CSpow}, truncated after 15 terms, are more accurate and efficient on the interval $[0,1.5]$, this conclusion endorsing recommendations in \cite{NR}.

Part of the motivation for this paper was a remark in Weideman \cite{JA} regarding the modified trapezium rule methods of \cite{MR,HR} for computing $\erfc(z)$, that they are
``very accurate, provided for given $z$ and $N$ [the finite number of quadrature points retained] the optimal stepsize $h$ is selected. It is not easy, however, to determine this optimal $h$ a priori.'' At least as far as computing $\erfc(z)$ for $\arg(z)= -\pi/4$ is concerned (which, by \eqref{Frelw}, is the same as computing $F(x)$) this problem is solved in this paper, so that the effectiveness of the modified trapezium rule methods of \cite{MR,HR,JA} is clearly demonstrated. We hope that the methodology and positive results of this paper will inspire further applications of this truncated, modified trapezium rule method.

We finish by flagging that the modified trapezium rule method that we have used in this paper is applicable widely to the evaluation of integrals on the real line of functions that are analytic but with poles near the real axis. Indeed, general theories of the method are presented in Bialecki \cite{bialecki}, Hunter \cite{hunter92} (and see \cite{crouch}, \cite[\S5.1.4]{prythe}), and  in the thesis of one of the authors \cite{laporte}, where the emphasis is on the particular case
\eqref{pc},
where the analytic function $f(t)=O(1)$ as $t\to \pm \infty$. Integrals of the form \eqref{pc} arise in probabilistic applications \cite{crouch} and  as representations in integral form of solutions to linear PDEs with constant coefficients, after solution by Fourier transform methods and deformation of the path of integration to a steepest descent path. One example which continues to be the subject of computational studies \cite{CWH95,nedelec,greengard} is the Green's function for the Helmholtz equation $\Delta u + k^2 u =0$ in a half-space with an impedance boundary condition, $\partial u/\partial n = \ri k \beta u$. Representations for this Green's function in terms of a steepest descent path integral of the form \eqref{pc}, in both the 2D and 3D cases, are given in \cite{CWH95}, and the application of the truncated modified trapezium rule method is discussed in \cite{laporte}.

\begin{acknowledgements}
This paper is dedicated to David Hunter, formerly of the University of Bradford, UK, who celebrated his 80th birthday in April 2013. Sadly David passed away on 15 August 2013. David was a kind and gentle man and a fine mathematician and teacher and the second author acknowledges his gratitude for David's contribution to his education as a numerical analyst at Bradford in the 80s. We also acknowledge the very helpful and thorough comments of the two anonymous referees.
\end{acknowledgements}

\appendix

\normalsize
\section{Appendix: Bounds on $\mathrm{erfc}$} \label{appendix}
In this appendix we prove Theorem \ref{lemma} as a corollary of bounds on $\erfc$ in the right hand complex plane contained in Theorem \ref{thm:erfc} below. In particular \eqref{Flb1} follows immediately from \eqref{Frelw} and the first bound in \eqref{erfclb}, while \eqref{Flb2} follows from \eqref{symm}, \eqref{Frelw}, and the second of the bounds \eqref{erfclb}. The bounds in Theorem \ref{thm:erfc} are well-known in the case $z\geq 0$ \cite[(7.8.2)-(7.8.3)]{NIST}, and the second bound (equivalent by \eqref{Frelw} to the bound $|w(z)| \leq 1$ for $\mathrm{Im}(z) \geq 0$) is recently proved by an alternative argument on p.~413 of \cite{arens}.
\begin{theorem} \label{thm:erfc}
For $z=x+\ri y$ with $x\geq 0$, $y\in \R$, we have that
\begin{equation} \label{erfclb}
|\erfc(z)|\geq \frac{\re^{y^2-x^2}}{\sqrt{(1+\sqrt{\pi}\,x)^2 + \pi y^2}}\geq \frac{\re^{y^2-x^2}}{1+\sqrt{\pi}\,|z|} \;\mbox{ and }\; |\erfc(z)| \leq \re^{y^2-x^2}.
\end{equation}
\end{theorem}
\begin{proof}
The first of the bounds \eqref{erfclb} is equivalent to the bound
\begin{equation} \label{erfclb2}
|\mathcal{G}(z)| \geq 1, \quad \mbox{for } \mathrm{Re}(z) \geq 0,
\end{equation}
where $\mathcal{G}(z) = (1+\sqrt{\pi}\, z) \re^{z^2}\erfc(z)$ is an entire function which has the properties that $\mathcal{G}(0)=1$ and $\mathcal{G}(z)\to 1$ as $|z|\to \infty$ in the right hand plane, uniformly in $\arg(z)$ \cite[(7.1.23)]{AS}. (These properties imply that the first of the bounds \eqref{erfclb} is sharp for $z=0$ and in the limit $|z|\to \infty$.) We will show \eqref{erfclb2} by showing that \eqref{erfclb2} holds for all $z$ in the right hand plane if it holds on the imaginary axis, and then showing that \eqref{erfclb2} holds on the imaginary axis.

To see that it is enough to prove that \eqref{erfclb2} holds for imaginary $z$, observe that, since $\erfc(z)$ has no zeros in the right hand complex plane \cite{ON,FE} (or on the imaginary axis where ${\mathrm{Re}}(\erfc(z)) = 1$, see \eqref{thing}), the function $\mathcal H(z) := 1/\mathcal{G}(z)$ is also analytic in the right hand complex plane and is continuous up to the imaginary axis. Moreover, $\mathcal{H}(z)$ is bounded in the right hand plane since, as observed above, $\mathcal{G}(z)\to 1$ as $|z|\to \infty$ in the right hand plane (uniformly in $\arg(z)$).  Since $\mathcal{H}(z)$ is bounded in the right hand plane, it follows from the maximum principle that
\begin{equation} \label{sup}
\sup_{\mathrm{Re}(z) \geq 0} |\mathcal{H}(z)| = \sup_{\mathrm{Re}(z) = 0} |\mathcal{H}(z)|.
\end{equation}
To see this, note that this equality holds for $\mathcal{H}_\alpha(z) := 1/\mathcal{G}_\alpha(z)$, with $\alpha >1$, where $\mathcal{G}_\alpha(z) := (1+\sqrt{\pi}\, z)^\alpha \re^{z^2}\erfc(z)$ with the branch cut taken as the negative real axis. This is clear since $\mathcal{H}_\alpha(z)$ is analytic in the right half-plane, continuous up to the imaginary axis, and vanishes at infinity, so that the standard maximum principle implies that  $\mathcal{H}_\alpha(z)$ takes its maximum value on the imaginary axis. But then \eqref{sup} follows by taking the limit $\alpha \to 1^+$.

In view of \eqref{sup}, to establish \eqref{erfclb2} we need only show that it holds for $z=\ri y$ with $y\in \R$; indeed, establishing this bound for $y\geq 0$ is sufficient since $\erfc(-\ri y)=\overline{\erfc(\ri y)}$.
Now, for $z=\ri y$ with $y\geq 0$, using \cite[(7.5.1)]{NIST}, which implies
\begin{equation} \label{thing}
\re^{z^2}\erfc(z)= \re^{-y^2}\left(1 -\frac{2\ri}{\sqrt{\pi}}\int_0^y \re^{t^2} \rd t\right)
\end{equation}
we see that
\begin{eqnarray} \label{norms}
|\mathcal{G}(\ri y)|^2 & =&	(1+\pi y^2) \re^{-2y^2}\left(1 + \frac{4}{\pi}\left(\int_0^y \re^{t^2} \rd t\right)^2\right)\\ \nonumber
& \geq & (1+\pi y^2) \re^{-2y^2}\left(1 +\frac{4}{\pi}y^2\right)\\\nonumber
&=& \left(1+ \left(\pi + \frac{4}{\pi}\right) y^2 + 4 y^4\right) \re^{-2y^2}.
\end{eqnarray}
It is an easy calculus exercise to show the right hand side takes its minimum value on $[0,1]$ at either 0 or 1, and hence that $|\mathcal{G}(\ri y)| \geq 1$, for $0\leq y \leq 1$, since $|\mathcal{G}(\ri)|^2 > (5+\pi)/\re^2>8/2.8^2>1$. Further, \eqref{norms} implies that
\begin{equation} \nonumber
|\mathcal{G}(\ri y)| \geq 2 y \re^{-y^2}\int_0^y \re^{t^2} \rd t
\end{equation}
and, for $y\geq 1$, it follows on integrating by parts that
\begin{eqnarray*}
\int_0^y \re^{t^2} \rd t  = \int_0^1 \re^{t^2} \rd t + \int_1^y e^{t^2} \rd t &= &\int_0^1\re^{t^2}\rd t + \frac{\re^{y^2}}{2y} - \frac{\re}{2} + \int_1^y \frac{\re^{t^2}}{2t^2}\rd t\\
 &> &\int_0^1(1+t^2 + \tfrac{1}{2}t^4)\rd t + \frac{\re^{y^2}}{2y} - \frac{\re}{2} > \frac{\re^{y^2}}{2y},
\end{eqnarray*}
since $\re < 2.8 < 2(1+1/3+1/10)$. Thus $|\mathcal{G}(\ri y)|\geq 1$ on $[1,\infty)$ and the bound \eqref{erfclb2} is proved.

Similarly,
\begin{equation} \label{t2}
\sup_{\mathrm{Re}(z) \geq 0} |\re^{-z^2}\erfc(z)| = \sup_{\mathrm{Re}(z) = 0} |\re^{-z^2}\erfc(z)| = \sup_{y\geq 0} |\re^{-y^2}\erfc(\ri y)|.
\end{equation}
Further, \eqref{thing} implies that, for $y\geq 0$,
\begin{eqnarray*}
|\erfc(\ri y)|^2-1 &=& \frac{4}{\pi}\left(\int_0^y \re^{t^2} \rd t\right)^2 = \frac{4y^2}{\pi}\left(\sum_{n=0}^\infty \frac{y^{2n}}{n!(2n+1)}\right)^2 \\
&=& \frac{2y^2}{\pi} \sum_{n=0}^\infty a_n y^{2n} \leq \frac{2}{\pi}\left(\re^{2y^2}-1\right)
\end{eqnarray*}
where
$$
a_n = \sum_{m=0}^n \frac{2}{m!(n-m)!(2m+1)(2(n-m)+1)} \leq \frac{2}{n+1}\sum_{m=0}^n\frac{1}{m!(n-m)!} = \frac{2^{n+1}}{(n+1)!}.
$$
Thus, for $y\geq 0$,
\begin{equation} \nonumber
|\re^{-y^2}\erfc(\ri y)|^2 \leq \frac{2}{\pi} + \left(1-\frac{2}{\pi}\right) \re^{-2y^2} \leq 1.
\end{equation}
Combining this with \eqref{t2} we see that the second of the bounds \eqref{erfclb} holds.
\end{proof}
\end{document}